\documentclass[a4paper]{article}

\usepackage[utf8]{inputenc}
\usepackage{lmodern}
\usepackage{geometry}
\usepackage{array}
\usepackage{graphicx}
\usepackage{amssymb, amsfonts, amsthm, amsmath}
\usepackage[english]{babel}
\usepackage[pdfborder={0 0 0}]{hyperref}
\usepackage{tikz}
\usetikzlibrary{decorations}
\usetikzlibrary{decorations.pathreplacing}
\tikzset{individu/.style={draw,thick}}

\usepackage{subfigure}

\theoremstyle{plain}
\newtheorem{theorem}{Theorem}[section]

\newtheorem{lemma}[theorem]{Lemma}

\theoremstyle{definition}

\newtheorem{remark}[theorem]{Remark}

\theoremstyle{remark}

\newcommand{\Z}{\mathbb{Z}}
\newcommand{\R}{\mathbb{R}}

\newcommand{\ind}[1]{\mathbf{1}_{\{#1\}}}

\numberwithin{equation}{section}

\DeclareMathOperator{\E}{\mathbf{E}}
\renewcommand{\P}{\mathbf{P}}

\newcommand{\calL}{\mathcal{L}}
\newcommand{\calU}{\mathcal{U}}

\newcommand{\calA}{\mathcal{A}}

\newcommand{\calN}{\mathcal{N}}
\renewcommand{\S}{\mathbf{S}}

\newcommand{\norm}[1]{{\left\Vert #1 \right\Vert}}

\title{Maximal displacement of $d$-dimensional branching~Brownian~motion}
\author{Bastien Mallein}
\date{\today}
\renewcommand{\tilde}[1]{\widetilde{#1}}
\renewcommand{\epsilon}{\varepsilon}
\renewcommand{\phi}{\varphi}
\begin{document}

\maketitle

\begin{abstract}
We consider a branching Brownian motion evolving in $\R^d$. We prove that the asymptotic behaviour of the maximal displacement is given by a first ballistic order, plus a logarithmic correction that increases with the dimension $d$. The proof is based on simple geometrical evidence. It leads to the interesting following side result: with high probability, for any $d \geq 2$, individuals on the frontier of the process are close parents if and only if they are geographically close.
\end{abstract}

\section{Introduction}
\label{sec:introduction}

Let $d \geq 1$. A \textit{$d$-dimensional branching Brownian motion} (or $d$-dim. BBM for short) is a particle process in which individuals displace according to independent Brownian motions and reproduce at rate 1. It starts with a unique individual positioned at $0 \in \R^d$ at time $t = 0$. This individual displaces according to a $d$-dimensional Brownian motion until time $T$, distributed as an exponential random variable independent of the displacement. At time $T$, the individual dies giving birth to two children on its current position. These two particles then start independent $d$-dimensional branching Brownian motions. 

For any $t \geq 0$, we write $\calN_t$ for the set of individuals alive at time $t$ in the process. Let $u \in \calN_t$ and $s \leq t$, we set $X_s(u)$ the position at time $s$ of the ancestor of $u$ that was alive at that time. The quantity of interest is $R_t = \max_{u \in \calN_t} \norm{X_t(u)}$. Bramson \cite{Bra78} proved the following asymptotic behaviour in dimension 1
\begin{equation}
  \label{eqn:bramson}
  R_t = \sqrt{2}t - \frac{3}{2\sqrt{2}} \log t + O_\P(1),
\end{equation}
where $O_\P(1)$ denotes a generic process $(Y_t, t \geq 0)$ such that $\lim_{K \to +\infty} \sup_{ t \geq 0} \P(|Y_t| \geq K) = 0$. This process has been intensively studied in dimension 1, partly because of its links with the FKPP equation: $\partial_t u = \frac{1}{2}\partial^2_x u + u(1-u)$. The function
\[
  u : (t,x)\in (0,+\infty) \times \R \longmapsto \P\left(\max_{u \in \calN_t} X_t(u) \geq x \right),
\]
is a travelling wave solution of the F-KPP equation.

Gärtner \cite{Gar81} studied a $d$-dimensional generalization of the F-KPP equation, solved by e.g.,
\[
  w : (t,x) \in (0,+\infty) \times \R^d  \longmapsto \P\left(\exists u \in \calN_t : \norm{X_t(u)-x} \leq 1 \right).
\]
It is proved that for large $t$, the function $w(t,.)$ admits a sharp cutoff located close to the ball of radius $\sqrt{2} t - \frac{d+2}{2\sqrt{2}} \log t$. Consequently the probability to find an individual $u \in \calN_t$ within distance $1$ of a given point $x$ is small if $\norm{x} \gg \sqrt{2} t - \frac{d+2}{2\sqrt{2}} \log t$ and large if $\norm{x} \ll \sqrt{2} t - \frac{d+2}{2\sqrt{2}} \log t$. Therefore if we replace every individual alive at time $t$ by a ball of radius 1 the radius $\rho_t$ of the percolation cluster containing the origin should verify $\rho_t = \sqrt{2}t - \frac{d+2}{2\sqrt{2}} \log t + O_\P(1)$. Observe the logarithmic correction decreases as the dimension increases.

However the projection of the $d$-dim. BBM on any given line is a $1$-dim. BBM. By \eqref{eqn:bramson}, with high probability the maximal displacement in the process is larger than $\sqrt{2} t - \frac{3}{2\sqrt{2}} \log t$. Consequently, $R_t$ has a different behaviour than $\rho_t$. In particular, we prove in Theorem \ref{thm:main} below that while the first order does not depend on the dimension, the logarithmic correction of $R_t$ increases with $d$. 

There have been few studies of quantities similar to $R_t$. In \cite{MRV15}, the authors consider a complex, thus $2$-dimensional, branching Brownian motion, but considered the process around its maximal value in 1 direction. Similarly, \cite{BoH14} also considered a model related to the $2$-dimensional branching Brownian motion, in which individuals diffuse as Brownian motions in one direction, and move at piecewise ballistic speed in the orthogonal direction. They described the extremal process, around the individual that travelled the farthest in the diffusive direction.

The main result of this article extends Bramson's result on the asymptotic behaviour of the maximal displacement $R_t$ to any dimension $d \geq 1$.
\begin{theorem}
\label{thm:main}
For any $d \geq 1$, we have
\[ R_t = \sqrt{2}t + \frac{d-4}{2\sqrt{2}} \log t + O_\P(1). \]
Moreover there exists $C>0$ such that for any $t \geq 1$ and $y \in [1,t^{1/2}]$,
\begin{equation}
  \label{eqn:main}
  \frac{ye^{-\sqrt{2}y}}{C} \leq \P\left( R_t \geq \sqrt{2}t + \frac{d-4}{2\sqrt{2}} \log t + y \right) \leq C ye^{-\sqrt{2}y}.
\end{equation}
\end{theorem}

In the rest of this article, $C$ stands for a generic positive constant, chosen large enough, that may change from line to line. Moreover, we write $x \wedge y$ for the minimum between $x$ and $y$ and $x_+$ as the maximum between $x$ and 0.

\begin{remark}
\label{rem:main}
As observed above, for any $v \in \S^{d-1}$ (the $d-1$-dimensional sphere), the process $((X_t(u).v, u \in \calN_t), t \geq 0)$ is a $1$-dim. BBM, thus $\max_{u \in \calN_t} X_t(u).v = \sqrt{2} t - \frac{3}{2\sqrt{2}} \log t + O_\P(1)$ by~\eqref{eqn:bramson}. We expect that for large times $t$, the convex hull of the set of occupied positions at time $t$ looks like a ball of radius $\sqrt{2} t - \frac{3}{2\sqrt{2}} \log t$, with spikes of height $\frac{d-1}{2\sqrt{2}} \log t$, tossed uniformly at random on the surface of the ball.
\end{remark}

The asymptotic behaviour of $R_t$ may be decomposed as follows:
\[
  R_t = \sqrt{2} t - \frac{3}{2\sqrt{2}} \log t + \frac{d-1}{2\sqrt{2}} \log t + O_\P(1).
\]
The term $-\frac{3}{2\sqrt{2}} \log t$ comes from the branching structure of the process. It is linked to the exponent of decay of the probability for a Brownian motion to make an excursion of length $t$ above 0\footnote{This is underlined in \cite{AiS10} for the closely related model of the branching random walk.}. The term $\frac{d-1}{2\sqrt{2}} \log t$ comes from the fact that the frontier of the $d$-dim. BBM is the $d-1$ dimensional sphere of radius $O(t)$. There are $O(t^{(d-1)/2})$ ``distinct'' directions that can be followed to reach the maximal displacement at time $t$., yielding a term similar to the maximum of $O(t^{(d-1)/2})$ independent exponential random variables with parameter $\sqrt{2}$.

The rest of the paper is organised as follows. In Section \ref{sec:lemmas}, we introduce the celebrated many-to-few lemma. We also recall some Brownian motion and geometric estimates, that precise the rough picture presented in Remark \ref{rem:main}. Section~\ref{sec:upperbound} is then devoted to the proof of the upper bound of \eqref{eqn:main}, that comes from a frontier argument; and Section~\ref{sec:lowerbound} to its lower bound, using second moment methods. We end Section~\ref{sec:lowerbound} completing the proof of Theorem \ref{thm:main}.

\section{Preliminary lemmas}
\label{sec:lemmas}

\subsection{The many-to-few lemmas}

Let $((X_t(u), u \in \calN_t),t \geq 0)$ be $d$-dim BBM. The many-to-one lemma links the mean of an additive function of the branching Brownian motion with a Brownian motion estimate. Its origins can be tracked back to the works of Peyrière \cite{Pey74} and Kahane and Peyrière \cite{KaP76}. This lemma has been enhanced, through the so-called spinal decomposition and stopping lines theory. However we only need in this article a simple version, corollary of \cite[Lemma 1]{HaR14}.

\begin{lemma}[Many-to-one lemma]
\label{lem:many-to-one}
For any $t \geq 0$ and measurable positive function $f$, we have
\[
  \E\left[ \sum_{u \in \calN_t} f(X_s(u), s \leq t) \right] = e^t \E\left[f(B_s, s \leq t) \right],
\]
where $B$ is a $d$-dimensional Brownian motion.
\end{lemma}
This lemma is used to bound the mean number of individuals belonging to certain specific sets. To bound from below the probability for a set of individuals to be non-empty, we compute some second moments. This result, sometimes called in the literature the many-to-two lemma is also a consequence of \cite[Lemma 1]{HaR14}.

\begin{lemma}[Many-to-two lemma]
\label{lem:many-to-two}
Let $B$ and $B'$ be two independent $d$-dimensional Brownian motions. For $s \geq 0$ we set $W^{(s)} : t \in [0,+\infty) \mapsto B_{t \wedge s} + B'_{(t-s)_+}$, a Brownian path identical to $B$ until time $s$ and with independent increments afterwards. For any measurable positive functions $f,g$ and $t \geq 0$, we have
\begin{multline*}
  \E\left[ \left(\sum_{u \in \calN_t} f(X_s(u), s \leq t)\right) \left( \sum_{u \in \calN_t} g(X_s(u), s \leq t) \right) \right]\\
  = \E\left[ \sum_{u \in \calN_t} f(X_s(u), s \leq t) g(X_s(u), s \leq t)\right] 
  + \int_0^t e^{2t-s} \E\left[ f(B_u, u \leq t) g(W^{(s)}_u, u \leq t) \right].
\end{multline*}
\end{lemma}

\subsection{Ballot theorem for the Brownian motion}

We recall in this section some well-known Brownian motion estimates. Let $\beta$ be a standard Brownian motion on $\R$. The quantity $I_t = \inf_{s \leq t} \beta_s$ has the same law as $-|\beta_t|$. Consequently there exists $C>0$ such that for any $t \geq 1$ and $y \geq 1$,
\begin{equation}
  \label{eqn:ballot}
  \frac{y \wedge t^{1/2}}{C t^{1/2}} \leq \P(\beta_s \geq -y, s \leq t) \leq \frac{C(y \wedge t^{1/2})}{t^{1/2}}.
\end{equation}
We often call this equation the ballot theorem, for its similarities with the well-known random walk estimate (see \cite{ABR08}).

Using the Girsanov theorem, it is an easy exercise to prove that for any $A \in \R$ and $\alpha < 1/2$, there exists $C>0$ such that for all $t \geq 1$ and $y \geq 1$,
\begin{equation}
  \label{eqn:ballotBended}
  \frac{y \wedge t^{1/2}}{Ct^{1/2}} \leq \P(\beta_s \geq -y + A s^\alpha, s \leq t) \leq \frac{C(y \wedge t^{1/2})}{t^{1/2}}.
\end{equation}

These estimates can be used to compute the probability for a Brownian motion to make an excursion above a given curve. Dividing the Brownian path on $[0,t]$ into three pieces, the first and last pieces being Brownian motion that stay above a given path, the middle part joining these two pieces, we obtain the following result\footnote{For a similar computation for random walks, see e.g. \cite{AiS10,Mal14a}.}. For any $A>0$ and $\alpha < 1/2$, there exists $C>0$ such that for any $t \geq 1$, for any function $f$ satisfying
\begin{equation}
  \label{eqn:propFunction}
  \sup_{s \leq t} \frac{|f(s)|}{s^\alpha}+\frac{|f(t)-f(s)|}{(t-s)^\alpha} < A,
\end{equation}
(which implies in particular $f(0)=f(t)=0$) and for any $y,z \geq 1$, we have
\begin{equation}
  \label{eqn:excursionBended}
  \frac{(y\wedge t^{1/2})(z \wedge t^{1/2})}{Ct^{3/2}} \leq \P\left(\begin{array}{l} \beta_s \geq -y + f(s), s \leq t\\ \beta_t + y - f(t) \in [z,z+1]\end{array} \right) \leq \frac{C(y \wedge t^{1/2}) (z \wedge t^{1/2})}{t^{3/2}}.
\end{equation}

\subsection{Geometry estimates}

We conclude the section with an observation on the geometry of the sphere
\[
  \S^{d-1} = \left\{ (x_1,\ldots x_d) \in \R^d : \norm{x} := x^2_1 + \ldots + x_d^2 = 1 \right\}.
\]
Using the fact that this is a manifold of dimension $d-1$, we are able to exhibit $t^{(d-1)/2}$ ``distinct'' directions on the sphere of radius $t$.

\begin{lemma}
\label{lem:geoUpperbound}
There exists $C>0$ such that for any $R>1$, there exists $\calU(R) \subset \S^{d-1}$ verifying $\# \calU(R) \leq C R^{(d-1)/2}$ and
\begin{equation}
  \label{eqn:def}
  \left\{ x \in \R^d : \norm{x} \geq R \right\} \subset  \bigcup_{ v \in \calU(R)} \left\{ x \in \R^d : x.v \geq R-1 \right\}.
\end{equation}
\end{lemma}

\begin{proof}
Let $R>1$ and $x \in \R^d$ such that $\norm{x} \geq R$. We set $y = R \frac{x}{\norm{x}}$ its projection on the sphere of radius $R$. We note that for any $v \in \S^{d-1}$, if $y.v \geq R-1>0$ then $x.v \geq R-1$.

Therefore, it is enough to prove there exists $\calU(R)$ such that
\begin{equation}
  \label{eqn:defprop}
  \left\{ x \in \R^d : \norm{x} = R \right\} \subset \bigcup_{ v \in \calU(R)} \left\{ x \in \R^d: \norm{x} = R, x.v \geq R-1 \right\}.
\end{equation}
Let $v \in \S^{d-1}$ and $x \in \R^d$, we have $\norm{x - Rv}^2 = \norm{x}^2 - 2 R x.v + R^2$, thus
\[
  \left\{ x \in \R^d : \norm{x} = R, x . v \geq R-1 \right\}
  = \left\{ x \in \R^d :  \norm{x} = R, \norm{x - R.v} \leq \sqrt{2R} \right\}.
\]

\begin{figure}[ht]
\centering
\begin{tikzpicture}[scale=1.6]
  \fill [color=black!20] (0,1.7) ellipse (1.05 and 0.1);
  \fill [color=black!20] (1.05,1.7) arc (58.2:121.8:2) -- cycle;
  \draw [thick] (-2,0) arc (180:0:2);
  \draw [thick, densely dotted] (-2,0) arc (180:190:2);
  \draw [thick, densely dotted] (2,0) arc (0:-10:2);
  \draw [color = black!70, dashed] (1.04,1.7) arc (0:180:1.04 and 0.1);
  \draw (1.04,1.7) arc (0:-180:1.04 and 0.1);
  \draw [dotted] (0,0) node {$\scriptscriptstyle \bullet$} -- (0,1.7) node {$\scriptscriptstyle\bullet$} -- (0,2) node {$\scriptscriptstyle\bullet$};
  \draw [<->] (0,1.71) -- (0,1.99);
  \draw [dotted] (0,1.7) -- (1.05,1.7) node {$\scriptscriptstyle\bullet$};
  \draw [->] (0,0) -- (0,0.3) node[right] {$v$};
  \draw (0,1.8) node [left] {$1$};
  \draw [<->] (0,1.5) -- (1.05,1.5);
  \draw (0.5,1.5) node [below] {$\sqrt{2R-1}$};
  \draw [->] (0,0) -- (1.936,0.5);
  \draw (0.95,0.25) node [above] {$R$};
\end{tikzpicture}
\caption{It takes $R^{(d-1)/2}$ spherical caps of height 1 to cover a sphere of radius $R$.}
\end{figure}
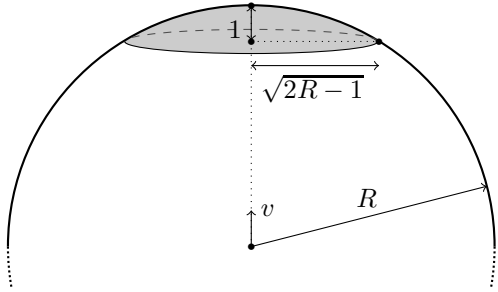

By homothetic transformation of ratio $R^{-1}$, a set $\calU(R)$ satisfying \eqref{eqn:defprop} satisfies
\[
  \S^{d-1} = \bigcup_{v \in \calU(R)} \left\{ u \in \S^{d-1} : \norm{v-u} \leq \sqrt{2/R} \right\}.
\]
As there exists $C>0$ such that for any $\epsilon>0$ small enough, the sphere $\S^{d-1}$ can be paved by $O(\epsilon^{1-d})$ balls of radius $\epsilon$, we write $\calU(R)$ for the set of the center of such a tiling with balls of radius $\sqrt{2/R}$.
\end{proof}

To explicitly construct a set $\calU(R)$ solution of \eqref{eqn:def}, one can take the union of the image of $R^{-1/2} \Z^d$ by the stereographic projection of the northern and the southern hemisphere of $\S^{d-1}$, which is a Lipschitz bijective mapping.

\section{The upper bound of Theorem \ref{thm:main}}
\label{sec:upperbound}

Let $(X_t(u), u \in \calN_t)_{t \geq 0}$ be a $d$-dim. BBM. We obtain in this section an upper bound for the tail of $R_t = \max_{u \in \calN_t} \norm{X_t(u)}$. We prove that for any $t \geq 1$ and $y \geq 1$, with high probability, there exists no individual exiting at some time $s \leq t$ the ball of radius
\begin{equation}
  \label{eqn:defFrontier}
  f^{t,y}_s = \sqrt{2}s + \frac{d-1}{2\sqrt{2}} \log (s+y) - \frac{3}{2\sqrt{2}} \log \frac{t+1}{t-s+1} + y.
\end{equation}
We set $\tilde{f}^{t,y}_s = f^{t,y}_s - \sqrt{2}s$. We observe that for any $\alpha < 1/2$ there exists $A>0$ such that for any $t \geq 1$, $(\tilde{f}^{t,y}_s - \tilde{f}^{t,y}_0, s \leq t)$ satisfies \eqref{eqn:propFunction}. The main result of the section is the following.
\begin{lemma}
\label{lem:frontier}
There exists $C>0$ such that for any $t \geq 1$ and $y \in [1,t^{1/2}]$,
\[
  \P\left[ \exists u \in \calN_t, \exists s\leq t : \norm{X_s(u)} \geq f^{t,y}_s \right] \leq C y e^{-\sqrt{2}y}.
\]
\end{lemma}

\begin{proof}
Let $t \geq 1$ and $y \in [1, t^{1/2}]$. To simplify notation, we assume that $t$ is an integer. For any $0 \leq k \leq t-1$, we set 
\[
  Z^{(t)}_k(y) = \sum_{u \in \calN_{k+1}} \ind{\exists s \in [k,k+1] : \norm{X_s(u)} \geq f^{t,y}_s} \ind{\norm{X_s(u)} \leq f^{t,y}_s, s \leq k}.
\]
By the Markov inequality and the many-to-one lemma, we have
\begin{multline}
  \P\left[ \exists s \leq t, u \in \calN_t : \norm{X_s(u)} \geq f^{t,y}(s) \right]\\
  \leq \sum_{k=0}^{t-1} \E\left(Z^{(t)}_k(y)\right) \leq \sum_{k=0}^{t-1} e^{k+1} \P\left[ \norm{B_s}\leq f^{t,y}_s , s \leq k,\exists r \in [k,k+1] : \norm{B_r} \geq f^{t,y}_r  \right] \label{eqn:domination},
\end{multline}
where $B$ is a $d$-dimensional Brownian motion.

As $s \mapsto f^{t,y}_s$ is increasing, applying the Markov property at time $k$ we have
\[ \P\left[ \norm{B_s}\leq f^{t,y}_s, s \leq k, \exists r \in [k,k+1] : \norm{B_r} \geq f^{t,y}_r \right] \leq \E\left(\phi_k\left(\sup_{s \leq 1} \norm{B_s}\right) \right),
\]
where $\phi_k : x \in [0,+\infty) \mapsto P\left[ \norm{B_k} \geq f^{t,y}_k - x-1, \norm{B_s} \leq f^{t,y}_s, s \leq k  \right]$. We now bound $\phi_k$ from above using Lemma \ref{lem:geoUpperbound}. There exists $C>0$ such that for any $k \leq t$ and $x \geq 0$, we have
\begin{align*}
  \phi_k(x) 
  &\leq \sum_{v \in \calU(f^{t,y}_k - x+1)} \P\left( B_k.v \geq f^{t,y}_k - x-1, B_s.v \leq f^{t,y}_s, s \leq k  \right)\\
  &\leq C \left( 1 + (f^{t,y}_k - x)_+\right)^{(d-1)/2} \P\left( \beta_k \geq f^{t,y}_k - x-1, \beta_s \leq f^{t,y}_s, s \leq k \right)\\
  &\leq C(1 + k + y)^{(d-1)/2}\P\left( \beta_k \geq f^{t,y}_k - x-1, \beta_s \leq f^{t,y}_s, s \leq k \right),
\end{align*}
where $\beta$ is a standard Brownian motion, that has the same law as $B.v$ for any $v \in \S^{d-1}$. Using the Girsanov transform then \eqref{eqn:excursionBended}, we have
\begin{align*}
  \P\left( \beta_k \geq f^{t,y}_k - x-1, \beta_s \leq f^{t,y}_s, s \leq k \right)
  &= \E\left[ e^{-\sqrt{2} \beta_k - k} \ind{\beta_k + \sqrt{2}k \geq f^{t,y}_k-x-1, \beta_s + \sqrt{2}s \leq f^{t,y}_s, s \leq k} \right]\\
  &\leq C e^{-k} e^{\sqrt{2} (x-\tilde{f}^{t,y}_k)} \P\left( \beta_k \geq \tilde{f}^{t,y}_k - x - 1, \beta_s \leq \tilde{f}^{t,y}_s, s \leq k \right)\\
  &\leq Ce^{-k} \frac{(t+1)^{3/2} e^{\sqrt{2} (x-y)}}{(k+y)^{(d-1)/2}(t-k+1)^{3/2}} \frac{(1+x)(y+\log y)}{(k+1)^{3/2}}.
\end{align*}

We conclude that for any $x \geq 0$ and $k \leq t$,
\[
  \phi_k(x) \leq C (1+x) y e^{-k} e^{\sqrt{2}(x-y)} \frac{(t+1)^{3/2}}{(k+1)^{3/2}(t-k+1)^{3/2}}.
\]
As $\sup_{s \leq 1} \norm{B_s}$ has Gaussian concentration, \eqref{eqn:domination} yields
\begin{equation*}
  \P\left[ \exists s \leq t, u \in \calN_t : \norm{X_s(u)} \geq f^{t,y}(s) \right]
  \leq C y e^{-\sqrt{2}y} \sum_{k=0}^{t-1} \frac{(t+1)^{3/2}}{(k+1)^{3/2}(t-k+1)^{3/2}}
  \leq C y e^{-\sqrt{2}y}.
\end{equation*}
\end{proof}

\begin{remark}
From the proof of Lemma \ref{lem:frontier}, we obtain that the mean number of individuals hitting the frontier $f^{t,y}$ between times $k$ and $t-k$ is bounded from above by $\frac{C y e^{-\sqrt{2}y}}{k^{1/2}}$.
\end{remark}

Using Lemma \ref{lem:frontier}, we bound from above the tail of the maximal displacement at time $t$.
\begin{lemma}
\label{lem:tailUpperbound}
There exists $C>0$ such that for any $t \geq 1$ and $y \in [1,t^{1/2}]$,
\[
  \P\left(R_t \geq \sqrt{2}t + \frac{d-4}{2\sqrt{2}} \log t + y \right) \leq C y e^{-\sqrt{2} y}.
\]
\end{lemma}

\begin{proof}
Let $t \geq 1$ and $y \in [1,t^{1/2}]$. By continuity of the paths followed by the individuals we have
\[
  \P(R_t \geq f^{t,y}_t) \leq \P\left(\exists u \in \calN_t : X_t(u) \geq f^{t,y}_t \right) \leq \P\left(  \exists u \in \calN_t, \exists s \leq t: X_s(u) \geq f^{t,y}_s \right),
\]
consequently Lemma \ref{lem:frontier} yields $\P(R_t \geq f^{t,y}_t) \leq C y e^{-\sqrt{2}y}$.

As $f^{t,y}_t = \sqrt{2}t + \frac{d-1}{2 \sqrt{2}} \log (t+y) - \frac{3}{2\sqrt{2}} \log (t + 1) + y$, for any $t \geq 1$ large enough, for any $y \in [10,t^{1/2}]$, we have $f^{t,y-5}_t \leq \sqrt{2}t + \frac{d-4}{2\sqrt{2}} \log t + y \leq f^{t,y+5}_t$, concluding the proof.
\end{proof}

\section{A local lower bound on the maximal displacement}
\label{sec:lowerbound}

For any $t \geq 1$ and $y \in [1,t^{1/2}]$, we set
\[
  \calA^{t,y} = \left\{ u \in \calN_t : \forall s \leq t, \norm{X_s(u)} \leq f^{t,y}_s \right\}.
\]
By Lemma \ref{lem:frontier}, we have $\calA^{t,y} = \calN_t$ with probability at least $1 - C y e^{-y}$. Let $v \in \S^{d-1}$, we introduce
\[
  \calA^{t,y}_v = \left\{ u \in \calA^{t,y} : X_t(u).v \geq f^{t,y}_t - 1 \right\},
\]
which is with high probability the set of individuals that made a large displacement at time $t$ in direction $v$. To obtain a lower bound for $R_t$, we bound from below the probability that $\calA^{t,y}_v$ is non-empty, using the Cauchy-Scharz inequality. We start bounding the mean of $\# \calA^{t,y}_v$.
\begin{lemma}
\label{lem:firstmoment}
There exists $C>0$ such that for any $t \geq 1$, $y \in [1,t^{1/2}]$ and $v \in \S^{d-1}$,
\[
  \frac{ye^{-\sqrt{2}y}}{Ct^{(d-1)/2}} \leq \E\left[ \#\calA^{t,y}_v \right]\leq \frac{Cye^{-\sqrt{2}y}}{t^{(d-1)/2}}.
\]
\end{lemma}

\begin{proof}
Let $t \geq 1$, $y \in [1,t^{1/2}]$ and $v \in \S^{d-1}$. The upper bound of $\E(\#\calA_v^{t,y})$ is a straightforward computation. Applying the many-to-one lemma, we have
\begin{align*}
  \E\left[ \# \calA^{t,y}_v \right]
  &= e^t \P\left[ B_t.v \geq f^{t,y}_t-1, \norm{B_s} \leq f^{t,y}_s, s \leq t \right]\\
  &\leq e^t \P\left[ \beta_t \geq f^{t,y}_t - 1, \beta_s \leq f^{t,y}_s, s \leq t \right],
\end{align*}
where $B$ is a $d$-dimensional Brownian motion, and $\beta = B.v$ is a one-dimensional one. By the Girsanov transform,
\[
  \E\left[ \# \calA^{t,y}_v \right] \leq \E\left[ e^{-\sqrt{2} \beta_t} \ind{\beta_t \geq \tilde{f}^{t,y}_t - 1, \beta_s  \leq \tilde{f}^{t,y}_s, s \leq t} \right]
  \leq C e^{-\sqrt{2}y} t^{(4-d)/2} \frac{y}{t^{3/2}},
\]
using \eqref{eqn:excursionBended}, which ends the proof of the upper bound.

The lower bound is obtained in a similar fashion. Let $B$ be a $d$-dimensional Brownian motion, we set $\beta=B.v$ and $B^\perp = B - \beta v$. Note that $B^\perp$ is a $d-1$-dimensional Brownian motion independent of $\beta$. Let $\alpha < 1/2$, there exists $\lambda > 0$ such that for any $1 \leq s \leq t$ and $y \geq 1$,
\[
  \left\{ x \in \R^d : \begin{array}{l} - 1/2 \leq x.v \leq f^{t,y}_s - 1/2 - (s\wedge (t-s))^\alpha\\ \norm{x - (x.v)v} \leq 1/2 + \lambda \min(s^{\alpha + 1/2},t^{1/2})\end{array} \right\}\\
  \subset \left\{ x \in \R^d : \norm{x} \leq f^{t,y}_s \right\}.
\]
Therefore, using Lemma \ref{lem:many-to-one}, we have
\begin{align*}
  \E\left[ \# \calA^{t,y}_v \right]
  &= e^t \P\left[ B_t.v \geq f^{t,y}_t-1, \norm{B_s} \leq f^{t,y}_s, s \leq t \right]\\
  &\geq e^t \P\left[ \beta_t \geq f^{t,y}_t - 1, -1 \leq \beta_s \leq f^{t,y}_s-1/2-(s\wedge (t-s))^\alpha, s \leq t \right]\\
  &\qquad \qquad \qquad \qquad\qquad \times \P\left[ \norm{B^\perp_s} \leq 1/2 + \lambda \min\left(s^{1/2+\alpha} , t^{1/2}\right), s \leq t \right].
\end{align*}
By standard Brownian estimates, we have
\[
  \inf_{t > 0} \P\left[ \norm{B^\perp_s} \leq 1 + \lambda \min\left(s^{1/2+\alpha} , t^{1/2}\right), s \leq t \right] > 0.
\]
Moreover, using once again the Girsanov transform and \eqref{eqn:excursionBended}, we have
\begin{multline*}
  e^t \P\left[ \beta_t \geq f^{t,y}_t - 1, -1 \leq \beta_s \leq f^{t,y}_s-1/2-(s\wedge (t-s))^\alpha, s \leq t \right]\\
  \geq \E\left[ e^{-\sqrt{2} \beta_t} \ind{\beta_t \geq \tilde{f}^{t,y}_t - 1, -1 \leq \beta_s \leq \tilde{f}^{t,y}_s-1/2-(s\wedge (t-s))^\alpha, s \leq t} \right]
  \geq \frac{e^{-\sqrt{2}y}}{C} t^{(4-d)/2} \frac{y}{t^{3/2}}.
\end{multline*}
We conclude that $\E\left[ \# \calA^{t,y}_v \right] \geq e^{-\sqrt{2} y} t^{(4-d)/2} \frac{y}{Ct^{3/2}}$.
\end{proof}

We now bound from above the second moment of $\#\calA_v^{t,y}$.
\begin{lemma}
\label{lem:secondmoment}
There exists $C>0$ such that for any $t \geq 1$, $y \in [1,t^{1/2}]$ and $v \in \S^{d-1}$,
\[
  \E\left[ (\# \calA^{t,y}_v )^2 \right] \leq \frac{Cye^{-\sqrt{2}y}}{t^{(d-1)/2}} .
\]
\end{lemma}

\begin{proof}
To compute this second moment, we use Lemma \ref{lem:many-to-two}. Let $B$ and $B'$ be two independent Brownian motions, and $W^s : r \in [0,t] \mapsto B _{r \wedge s} + B'_{(r-s)_+}$. We have
\begin{equation}
  \label{eqn:secondapprox}
  \E\left[ \left(\# \calA^{t,y}_v\right)^2 \right] \leq \E\left[ \# \calA^{t,y}_v \right]
  + \int_0^t e^{2t-s}\P\left[ \begin{array}{l} B_t.v \geq f^{t,y}_t-1, \norm{B_r} \leq f^{t,y}_r, r \leq t \\ W^s_t.v \geq f^{t,y}_t - 1, \norm{W^s_r} \leq f^{t,y}_r, r \leq t\end{array} \right] ds.
\end{equation}
By Lemma \ref{lem:firstmoment}, we have $\E\left[ \# \calA^{t,y}_v \right] \leq Cye^{-\sqrt{2}y}t^{-(d-1)/2}$.

Let $0 \leq s \leq t$ and $\beta$ a standard Brownian motion. By the Markov property, we have
\begin{align*}
  e^{2t-s}\P\left[ \begin{array}{l} B_t.v \geq f^{t,y}_t-1, \norm{B_r} \leq f^{t,y}_r, r \leq t \\ W^s_t.v \geq f^{t,y}_t - 1, \norm{W^s_r} \leq f^{t,y}_r, r \leq t\end{array} \right]
  &\leq e^{2t-s} \P\left[ \begin{array}{l} B_t.v \geq f^{t,y}_t-1, B_r.v \leq f^{t,y}_r, r \leq t \\ W^s_t.v \geq f^{t,y}_t - 1, W^s_r.v \leq f^{t,y}_r, r \leq t\end{array} \right]\\
  &\leq e^s \E\left[ \phi_s(\beta_s)^2 \ind{\beta_r \leq f^{t,y}_r, r \leq s}\right],
\end{align*}
where $\phi_s : x \in \R \mapsto e^{t-s}\P\left[ \beta_{t-s} \geq f^{t,y}_t-1, \beta_r + x \leq f^{t,y}_{s+r}, r \leq t-s \right]$. By the Girsanov transform and \eqref{eqn:excursionBended} again, we have
\[
  \phi_s(x) \leq \frac{C(1 + (f^{t,y}_s - x)_+)e^{\sqrt{2}(x-y-\sqrt{2}s)}}{(t+1)^{(d-4)/2}(t-s)^{3/2}}.
\]
Consequently
\begin{align*}
  &e^{2t-s}\P\left[ \begin{array}{l} B_t.v \geq f^{t,y}_t-1, \norm{B_r} \leq f^{t,y}_r, r \leq t \\ W^s_t.v' \geq f^{t,y}_t - 1, \norm{W^s_r} \leq f^{t,y}_r, r \leq t\end{array} \right]\\
  &\qquad \qquad \leq \frac{Ce^{-2\sqrt{2}y}}{(t+1)^{d-4} (t-s+1)^3} e^s \E\left[ e^{2 \sqrt{2}(\beta_s-\sqrt{2}s)} \left(1 + \left(f^{t,y}_s - \beta_s\right)\right)^2 \ind{\beta_r \leq f^{t,y}_r, r \leq s} \right]\\
  &\qquad \qquad \leq \frac{Ce^{-2\sqrt{2}y}}{(t+1)^{d-4} (t-s+1)^3} \E\left[ e^{\sqrt{2}\beta_s} \left(1 + \left(\tilde{f}^{t,y}_s-\beta_s\right) \right)^2  \ind{\beta_r \leq \tilde{f}^{t,y}_r, r \leq s} \right],
\end{align*}
using the Girsanov transform. By decomposition with respect to the value of $\beta_s$ we have
\begin{align*}
  &\E\left[ e^{\sqrt{2}\beta_s} \left(1 + \left(\tilde{f}^{t,y}_s-\beta_s\right) \right)^2  \ind{\beta_r \leq \tilde{f}^{t,y}_r, r \leq s} \right]\\
  &\qquad \qquad \leq C \sum_{k = 0}^{+\infty} e^{\sqrt{2}(\tilde{f}^{t,y}_s-k)}(k+1)^2 \P\left( \beta_s-\tilde{f}^{t,y}_s \in [-k-1,-k], \beta_r \leq \tilde{f}^{t,y}_r, r \leq s \right)\\
  &\qquad \qquad \leq \frac{Cye^{\sqrt{2}y}}{(s+1)^{3/2}} \frac{(s+y+1)^{(d-1)/2}(t-s+1)^{3/2}}{t^{3/2}} \sum_{k=0}^{+\infty} (k+1)^3e^{-\sqrt{2}k}.
\end{align*}

We conclude that for any $0 \leq s \leq t$,
\begin{multline}
  \label{eqn:interest}
  e^{2t-s}\P\left[ \begin{array}{l} B_t.v \geq f^{t,y}_t-1, \norm{B_r} \leq f^{t,y}_r, r \leq t \\ W^s_t.v' \geq f^{t,y}_t - 1, \norm{W^s_r} \leq f^{t,y}_r, r \leq t\end{array} \right]\\ \leq \frac{Cye^{-\sqrt{2}y}}{(t+1)^{(d-1)/2}} \frac{(s+y+1)^{(d-1)/2}}{(t+1)^{(d-1)/2}} \frac{t^{3/2}}{(s+1)^{3/2}(t-s+1)^{3/2}}.
\end{multline}
Therefore, \eqref{eqn:secondapprox} yields $\E\left[ (\# \calA^{t,y}_v )^2 \right] \leq \frac{Cye^{-\sqrt{2}y}}{t^{(d-1)/2}}.$
\end{proof}

\begin{remark}
\label{rem:nice}
Observe that \eqref{eqn:interest}, a straightforward adaptation of Lemma \ref{lem:many-to-two}, Lemma~\ref{lem:geoUpperbound} and Lemma \ref{lem:frontier} yield the following noticeable side result: as soon as $d \geq 2$, for any $0 \leq R \leq t$
\begin{equation}
  \label{eqn:nice}
  \P\left[\exists u,u' \in \calN_t : \begin{array}{c} \norm{X_t(u)} \geq f^{t,0}_t, \norm{X_t(u')} \geq f^{t,0}_t, \norm{X_t(u)-X_t(u')} \leq t^{1/2}\\ \text{MRCA}(u,u') \leq t-R\end{array} \right] \leq \frac{C \log R}{R^{1/2}},
\end{equation}
where $\text{MRCA}(u,u')$ is the time at which the most recent common ancestor of $u$ and $u'$ was alive. In other words, in a $d$-dim. BBM, two individuals on the frontier of the process that are close to each other are close relatives. This result is well-known to fail in dimension 1, where individuals close to the edge of the process are either close relative, or their lineage had split within time $O_\P(1)$.
\end{remark}

Using Lemmas \ref{lem:firstmoment} and \ref{lem:secondmoment} as well as the Cauchy-Schwarz inequality, for any $t \geq 1$, $y \in [1,t^{1/2}]$ and $v \in \S^{d-1}$, we have
\begin{equation}
  \label{eqn:proba}
  \P\left( \exists u \in \calN_t : V(u).v \geq \sqrt{2}t + \frac{d-4}{2\sqrt{2} \log t} + y \right) \geq \frac{ye^{-\sqrt{2}y}}{C t^{(d-1)/2}}.
\end{equation}
This is a lower bound of the probability there exists an individual making a large displacement in direction $v$. To conclude the proof of the lower bound of Theorem \ref{thm:main}, we bound from above the correlation between the existence of individuals making large displacements in two distinct directions $v$ and $v'$ at the same time.
\begin{lemma}
\label{lem:crossedmoment}
There exists $C > 0$ such that for any $t \geq 1$, $y \in [1,t^{1/2}]$ and $v,v' \in \S^{d-1}$ such that $\norm{v-v'} \geq Ct^{-1/2}$ and $\norm{v+v'} \geq 3/2$, we have
\[
  \E\left[ \# \calA^{t,y}_v \# \calA^{t,y}_{v'} \right] \leq \frac{Cye^{-\sqrt{2}y}}{(t+1)^{(d-1)/2}} \left[ \frac{1}{\theta^{d-2}(t+1)^{(d-1)/2}} + e^{-\theta^2 t} \right],
\]
where $\theta = \arccos(v.v')$.
\end{lemma}

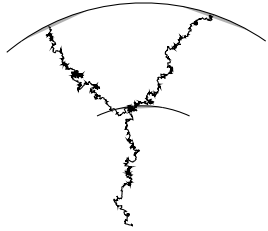
\begin{figure}[ht]
\centering
\begin{tikzpicture}[scale=0.4]
\fill [color=black!30] (-3.7,7.7) arc (125:105:7) -- cycle;
\fill [color = black!30] (3.8,8.03) arc (60:80:7) -- cycle;
\fill [color=black!30] (0.93,5.5) arc (80:105:3.6) -- cycle;
\draw (0.93,5.5) arc (80:65:3.6);
\draw (0.93,5.5) arc (80:115:3.6);

\draw (-0.0471336979397,1.56454565359) --(-0.063469715979,1.61154835108) --(-0.0568166713905,1.60514674446) --(-0.0635243768575,1.58552444894) --(-0.0796870951769,1.60768278472) --(-0.104634902552,1.6025396713) --(-0.158290557453,1.60706711765) --(-0.22255738734,1.63289638727) --(-0.265597491009,1.63676940028) --(-0.301609634219,1.65861129735) --(-0.279813214836,1.67534458395) --(-0.298054014779,1.71792499842) --(-0.281839449278,1.73041541582) --(-0.29808720058,1.72145284189) --(-0.279590055315,1.78250596912) --(-0.224806001708,1.74684584282) --(-0.228824294773,1.79003663809) --(-0.223721621593,1.87179480437) --(-0.178256294131,1.89973012039) --(-0.151345938096,1.93631052276) --(-0.160576343234,1.95231261938) --(-0.188393813171,1.96459916678) --(-0.20802171358,1.98012923155) --(-0.162077627035,1.97159712814) --(-0.186411683614,1.94123269975) --(-0.209272424955,1.93000504535) --(-0.262402665253,1.91586289768) --(-0.244133770319,1.90670913366) --(-0.231836152601,1.93635118276) --(-0.31005336467,1.95471695403) --(-0.331204329233,2.03305681745) --(-0.319237010616,2.04223392889) --(-0.253938865329,2.05766401593) --(-0.312253426372,2.09211274133) --(-0.350990528309,2.11327924173) --(-0.433259300325,2.12805937901) --(-0.456510049218,2.14396949192) --(-0.403668670191,2.17640235651) --(-0.404325441555,2.24447815433) --(-0.498366475979,2.23400126599) --(-0.539214649779,2.24121017356) --(-0.533120216171,2.23362663163) --(-0.529270613385,2.20151213287) --(-0.525441443436,2.22494457626) --(-0.58152281932,2.27547040149) --(-0.566720541519,2.29807326001) --(-0.578445407239,2.32908466824) --(-0.589837139917,2.31566150191) --(-0.589915721933,2.3241797858) --(-0.615650338459,2.34266411428) --(-0.624990061988,2.36685337964) --(-0.636005912535,2.36802918985) --(-0.664908964294,2.4369198989) --(-0.657049882044,2.45050285164) --(-0.645388677229,2.41991225055) --(-0.617542330816,2.42492557909) --(-0.651942018779,2.41503181555) --(-0.594055762901,2.4237759193) --(-0.586620225465,2.39673804228) --(-0.563556718358,2.41177448097) --(-0.535382517445,2.42876935111) --(-0.551772455817,2.44528803175) --(-0.561645389172,2.4189946697) --(-0.552500957982,2.46660118931) --(-0.525074417659,2.48554609334) --(-0.563614269588,2.5663670273) --(-0.543825530701,2.52463645483) --(-0.537598571864,2.4902453206) --(-0.483735831053,2.48821182987) --(-0.479244818848,2.47085713161) --(-0.456405695395,2.48305553484) --(-0.441400794364,2.49493236919) --(-0.416025820498,2.47630268211) --(-0.392064584879,2.53833102385) --(-0.403852679551,2.56991136922) --(-0.396234118067,2.60476178213) --(-0.392881044928,2.6758639946) --(-0.393867799831,2.69873392486) --(-0.379919130658,2.73287885712) --(-0.423649726021,2.7409248437) --(-0.39626989744,2.7946182679) --(-0.39891407605,2.86346915935) --(-0.357645185663,2.89210420371) --(-0.337221652881,2.90287411434) --(-0.336891689487,2.88838086504) --(-0.322961188782,2.91691514718) --(-0.290618096027,2.93586286774) --(-0.283357168412,2.93547603615) --(-0.295395115325,2.96491051732) --(-0.321846327039,2.95884767411) --(-0.281041431418,2.96002813305) --(-0.282482252148,2.96553237226) --(-0.263038669358,2.98016526168) --(-0.294193191238,2.96866019657) --(-0.247602142483,2.96768111195) --(-0.21252981477,2.94000614988) --(-0.138776292916,2.91944341322) --(-0.123738699248,2.93740011409) --(-0.125174122927,2.90822898287) --(-0.134652983952,2.94342920735) --(-0.103629518474,2.95760535547) --(-0.111398013594,2.97679567966) --(-0.134403485036,2.97211716277) --(-0.148800273519,2.98264454482) --(-0.127010517106,2.94396232604) --(-0.151982279686,2.92218945684) --(-0.185348776256,2.94173995495) --(-0.1601695392,2.92117766723) --(-0.162943626321,2.95972157244) --(-0.18084116143,2.98154419715) --(-0.168285507275,2.97207181248) --(-0.148133703134,2.98777905723) --(-0.15524998363,2.99695624207) --(-0.159728445368,2.93181051368) --(-0.158663788479,2.96499124038) --(-0.161371021514,2.98794063873) --(-0.132042878464,2.97856279759) --(-0.120623755366,2.99790729442) --(-0.0818955628793,3.04809676697) --(-0.112324894041,3.01393939083) --(-0.0866465108577,3.01819510448) --(-0.130489030438,3.02590690562) --(-0.0973817246496,3.03464910735) --(-0.0757916412691,3.04266206693) --(-0.121315326633,3.06192390295) --(-0.070381404304,3.12067907009) --(-0.10220581023,3.12936448139) --(-0.150463751497,3.16041955424) --(-0.104818315523,3.1592863927) --(-0.102826846036,3.15222233841) --(-0.0987090553727,3.15461665685) --(-0.150320589943,3.17811254821) --(-0.16965516802,3.20189942283) --(-0.188727256551,3.21115950615) --(-0.225204192009,3.18476495825) --(-0.234155934103,3.22966977654) --(-0.26616892895,3.25827545484) --(-0.23689927385,3.28039389511) --(-0.208746517168,3.29609683712) --(-0.19306725654,3.27124730501) --(-0.185961798501,3.3243703399) --(-0.147974533225,3.28256945156) --(-0.105987774446,3.30938797571) --(-0.106449159351,3.23804244507) --(-0.0954803346605,3.25542417952) --(-0.0635910209727,3.25919390438) --(-0.0864064074825,3.23537582985) --(-0.101283199386,3.24235805984) --(-0.203930653682,3.29598594361) --(-0.164739408758,3.34256660933) --(-0.17532776708,3.31424049629) --(-0.141083088595,3.33720271987) --(-0.220453137353,3.31086782577) --(-0.244575630696,3.31348454953) --(-0.188860427228,3.29415640412) --(-0.247814092697,3.3556043084) --(-0.222555862127,3.31992216706) --(-0.232628411127,3.32857868927) --(-0.263892186608,3.36098849714) --(-0.233752092416,3.3530701148) --(-0.236881583453,3.33269010158) --(-0.245598705638,3.34309202479) --(-0.260396636775,3.36718403439) --(-0.250921390519,3.3607351382) --(-0.181369897936,3.39006075569) --(-0.17703604346,3.39387485328) --(-0.199245355671,3.41692322518) --(-0.210878261037,3.42049315958) --(-0.156184708186,3.42566098478) --(-0.15355339089,3.46253200122) --(-0.110439192303,3.49217377028) --(-0.117360245375,3.52692898442) --(-0.180605363213,3.4988854172) --(-0.213059189652,3.45284946434) --(-0.163300404004,3.40161317707) --(-0.152569691087,3.429703351) --(-0.168482092092,3.41220891356) --(-0.16952966116,3.46204124165) --(-0.14762532641,3.50630386348) --(-0.0719550779068,3.52862802769) --(-0.0927880642945,3.52744035223) --(-0.121966326655,3.4839210714) --(-0.119602013821,3.48929147315) --(-0.147215593572,3.4777118637) --(-0.146360237254,3.43599513373) --(-0.141793318597,3.43341181029) --(-0.193544222809,3.46575949299) --(-0.169650405274,3.47553613828) --(-0.137608291422,3.50039860643) --(-0.152741733712,3.51710013549) --(-0.140626623372,3.53453825749) --(-0.14619421844,3.52498455648) --(-0.15392287263,3.52134750398) --(-0.190494117601,3.5587245485) --(-0.198381424383,3.55872114825) --(-0.296693511487,3.568440217) --(-0.273782803479,3.61497470332) --(-0.28691458135,3.64643478566) --(-0.319240596815,3.64534597965) --(-0.283975367552,3.62224849241) --(-0.237678949561,3.60789106701) --(-0.231797258331,3.64335757058) --(-0.234216392532,3.64241886933) --(-0.208683724366,3.69590951841) --(-0.212437606743,3.69359758523) --(-0.2373791997,3.72029025047) --(-0.203328861897,3.69317626126) --(-0.226180915183,3.73606668837) --(-0.259508600161,3.74631943831) --(-0.231813353898,3.73579990101) --(-0.25487604552,3.74715777196) --(-0.278293390352,3.71603615261) --(-0.28485075229,3.70896126006) --(-0.23364078822,3.7600590693) --(-0.223984438114,3.82196938542) --(-0.191425429625,3.79507835579) --(-0.196830799619,3.80804622399) --(-0.190567749152,3.82175588556) --(-0.18318896239,3.76068518853) --(-0.150346857128,3.73787569895) --(-0.133311479292,3.73289221174) --(-0.101176841228,3.70891392341) --(-0.0364959881128,3.68814403126) --(-0.0224982451758,3.71455645715) --(0.0181094073741,3.72397166897) --(0.0690543397182,3.78291094099) --(0.0509417491206,3.80712177788) --(0.0759246278663,3.84810217832) --(0.0428091837004,3.90174052189) --(0.0344193433716,3.88444545092) --(0.043145180335,3.91393180508) --(0.0550819803393,3.98129870211) --(0.0604150698525,3.95052073295) --(0.0503747585966,3.98419960592) --(0.0860810729991,4.00635774688) --(0.0088415866468,3.98722027811) --(0.0195641495912,3.98842389778) --(-0.00184010243906,3.92163262092) --(0.02643058675,3.91495723546) --(0.0309255686447,3.93512711978) --(0.0828660752424,3.99008453583) --(0.114201091929,4.05663681382) --(0.171706159123,4.07219829714) --(0.154224893883,4.09425476499) --(0.131411891748,4.08373392614) --(0.128781287093,4.07855472354) --(0.165372800621,4.04022501812) --(0.143829245854,4.06368114586) --(0.0713735445508,4.09330309175) --(0.0656097098204,4.06670063084) --(0.0281878915059,4.079289496) --(-0.00457620143839,4.10207573283) --(-0.0200078751153,4.12149093101) --(0.0134287105303,4.1636584473) --(-0.0474923111867,4.17302038241) --(-0.00824038787889,4.16405759112) --(0.0439159544811,4.15895146137) --(0.0254498124816,4.20467071865) --(-0.0093107166571,4.21142882819) --(-0.0351000793133,4.20327767667) --(-0.0325853030857,4.22781817052) --(-0.063457855436,4.21892690487) --(-0.065637569975,4.22904470435) --(-0.0861974017552,4.17188219459) --(-0.0599210321628,4.19634150772) --(-0.0649555682599,4.22458923071) --(-0.060175656848,4.22680104021) --(-0.0763219072707,4.1855098845) --(-0.0543377990006,4.21168380682) --(-0.0264845708444,4.20619393698) --(-0.0451979229082,4.1918219549) --(-0.107220374096,4.18212649483) --(-0.0897515699298,4.19931089624) --(-0.0887077562945,4.21257525158) --(-0.0500974859581,4.19028252012) --(-0.0833142833459,4.19420326059) --(-0.0703909715795,4.19887322352) --(-0.073852238681,4.22929912508) --(-0.0284499893325,4.21576368711) --(-0.0432332990081,4.26568375393) --(-0.0251628474977,4.2649886357) --(-0.0319833034429,4.3264798729) --(-0.0214980532322,4.29578755167) --(-0.0906956179434,4.30329084312) --(-0.1541597165,4.32985087727) --(-0.117404140095,4.36594538536) --(-0.158428770669,4.36607427374) --(-0.159656590012,4.36713609704) --(-0.180275658581,4.39432680868) --(-0.153280850325,4.38537759542) --(-0.111291233687,4.4224362908) --(-0.136955826588,4.42995219416) --(-0.150001281915,4.41733467566) --(-0.13454901032,4.39417244957) --(-0.118236442053,4.38964596028) --(-0.0846094334274,4.40374759705) --(-0.0966516909728,4.4078873908) --(-0.11341402383,4.41571055429) --(-0.122302602842,4.41378895938) --(-0.159052051987,4.44425257216) --(-0.170050564987,4.45687979157) --(-0.184515628991,4.43781072717) --(-0.181966233117,4.44860601978) --(-0.220825761807,4.4250277321) --(-0.227982095284,4.43223752044) --(-0.213625087298,4.40660787588) --(-0.187812316026,4.45292775574) --(-0.155419749541,4.493880495) --(-0.126518665247,4.58072699779) --(-0.152452476185,4.5977996836) --(-0.14640110002,4.62434591458) --(-0.196165277714,4.63678293488) --(-0.190473025609,4.64414964152) --(-0.200378792625,4.6431664146) --(-0.201473454429,4.63011475907) --(-0.171300769669,4.64999663761) --(-0.210312906163,4.62433502654) --(-0.188228857711,4.65643786657) --(-0.22080456177,4.64253955237) --(-0.166102426279,4.66935725661) --(-0.196753589213,4.754956423) --(-0.236690605604,4.77528495923) --(-0.268565194826,4.8090832732) --(-0.274949757487,4.77164255533) --(-0.297451579783,4.77528001369) --(-0.266821841152,4.81830960654) --(-0.254288138121,4.87481088528) --(-0.253857252483,4.87461183521) --(-0.270442693952,4.84646511555) --(-0.253289033376,4.84238427505) --(-0.245429062428,4.87095431511) --(-0.23833327904,4.85293967464) --(-0.23160207739,4.89468152125) --(-0.21556795878,4.89543843717) --(-0.219218566306,4.90833210581) --(-0.228667791637,4.9135238543) --(-0.224280230634,4.93021291925) --(-0.272894236919,4.97457081171) --(-0.273087127348,4.97619446855) --(-0.276001702644,4.98779649368) --(-0.272040864592,5.00148604572) --(-0.234065866537,4.96148384456) --(-0.261677184609,4.99420921636) --(-0.246456890027,5.02703190855) --(-0.269440058155,5.0897300271) --(-0.241162333286,5.08357744581) --(-0.247331692238,5.10748506521) --(-0.213042421607,5.10680976161) --(-0.180105290409,5.05679574734) --(-0.186690563184,5.05434996435) --(-0.195194517874,5.0480383457) --(-0.143204880047,5.11056350518) --(-0.166005168042,5.11037700791) --(-0.162180000105,5.07100048786) --(-0.167303345499,5.07351462408) --(-0.148197619678,5.08405345593) --(-0.130552018046,5.1088563729) --(-0.120157986196,5.09865035859) --(-0.096219784427,5.11885506626) --(-0.0357141342391,5.09047424386) --(-0.0470352069151,5.17260434641) --(-0.0628972764035,5.21189414282) --(-0.0716230185167,5.22247416567) --(-0.0748983020168,5.23670229394) --(-0.0778171204084,5.21438303483) --(-0.0880624494687,5.23168685871) --(-0.133663092642,5.28707600594) --(-0.140338838042,5.35593304535) --(-0.176100693015,5.39158149239) --(-0.136534219265,5.39016300271) --(-0.127390743424,5.39448318426) --(-0.105984008524,5.38994357344) --(-0.110235117849,5.43394713856) --(-0.0924169932345,5.43388359052) --(-0.101201671794,5.42635874779) --(-0.129688265081,5.46025055157)  ; 
\draw (-0.129688265081,5.46025055157) --(-0.137835845961,5.44945690694) --(-0.169530746015,5.40188263094) --(-0.184493224503,5.41173032523) --(-0.180263540467,5.41908596914) --(-0.189525138174,5.44852870551) --(-0.144946517879,5.4403075913) --(-0.149718257913,5.45915977537) --(-0.103376179201,5.42851291632) --(-0.0870500498002,5.40179599923) --(-0.0714173643659,5.409514481) --(0.00819197775937,5.46192982599) --(-0.0049360474213,5.49085645995) --(-0.0322823737127,5.42968077663) --(-0.0419602368851,5.42329970757) --(-0.0636556802291,5.45956542581) --(-0.0861204263057,5.51048662871) --(-0.121084962967,5.52682396381) --(-0.131553246723,5.53229279111) --(-0.135609067779,5.51323209021) --(-0.118980753463,5.5160885252) --(-0.148205839466,5.50087597891) --(-0.15327386707,5.45909937566) --(-0.149124713613,5.46712347105) --(-0.105145380659,5.41293393546) --(-0.107824614204,5.41581504252) --(-0.11661863327,5.42395144159) --(-0.119525890783,5.42811675919) --(-0.131612759816,5.44983369759) --(-0.129493324822,5.49858092557) --(-0.146405133569,5.51998662094) --(-0.129940412973,5.50981279454) --(-0.115351060757,5.48726429724) --(-0.110799702348,5.45847114403) --(-0.0933986281031,5.4850993941) --(-0.118971976443,5.51617021898) --(-0.176983388326,5.49120140458) --(-0.118703078345,5.46206535153) --(-0.072577972695,5.48348651768) --(-0.0716800128322,5.46155618522) --(-0.0652339830329,5.41596395741) --(-0.0766067815061,5.35860868653) --(-0.0897011980139,5.37748279688) --(-0.11775590968,5.37751711411) --(-0.139782284224,5.43147927391) --(-0.183547975789,5.44590354028) --(-0.196897760201,5.40580030297) --(-0.166131126812,5.42998682917) --(-0.163067487049,5.46951435312) --(-0.144511746319,5.49881329831) --(-0.181623802603,5.46469782339) --(-0.179444945547,5.45751780276) --(-0.149689288542,5.44351779583) --(-0.118820977101,5.38599819144) --(-0.151960327142,5.3437417121) --(-0.169562095314,5.33671349176) --(-0.176543373698,5.33434916848) --(-0.142489549814,5.3618438164) --(-0.135953777924,5.40423134861) --(-0.0557654225195,5.45861604437) --(-0.0217368636365,5.3965454259) --(-0.0303505833474,5.41958417688) --(-0.00805864612454,5.44936931349) --(-0.0575906645176,5.44555858915) --(-0.0597037976291,5.44382648908) --(-0.0279360819684,5.46035004655) --(-0.043949257202,5.44210191757) --(-0.0313515851426,5.41719639303) --(-0.0422109323709,5.44693667115) --(-0.0467318334501,5.48796539963) --(-0.0251914536974,5.42562719686) --(-0.00406496799233,5.37718518457) --(0.0425720720965,5.40777316462) --(0.0787007371743,5.40349962899) --(0.0954933420697,5.38086314544) --(0.0680974490681,5.43298716473) --(-0.00777438863359,5.45239268858) --(0.0229499923129,5.48690771312) --(0.0422646594003,5.51490178325) --(0.0745224367101,5.49865243952) --(0.0859082787613,5.59379715354) --(0.100107989199,5.62683521657) --(0.112314563158,5.61853785219) --(0.0801158668547,5.62289562425) --(0.0478803952849,5.65865840476) --(0.0367985521898,5.61089347239) --(0.0353160533149,5.64076399051) --(0.0241540022786,5.67704497743) --(0.0204158750681,5.65521015045) --(0.00745301527477,5.6498059526) --(0.048020275031,5.60887665482) --(0.0360105138625,5.58686832012) --(0.0292993252416,5.55078529388) --(0.0513134508393,5.52174219006) --(0.0765512007375,5.54526187941) --(0.0575227117274,5.61057263253) --(0.118820073195,5.58052252975) --(0.112256784848,5.576521536) --(0.13660184597,5.58065720988) --(0.107739498023,5.58738987095) --(0.184578656748,5.62499970146) --(0.146357148579,5.65280619662) --(0.165097339663,5.70388282592) --(0.149921812951,5.67196841551) --(0.232345526318,5.67035373445) --(0.187240721369,5.68777714057) --(0.244349087713,5.64981405048) --(0.236764557822,5.64985036339) --(0.220521433298,5.64046576439) --(0.271815722444,5.66375233553) --(0.326704734494,5.66090930422) --(0.32713609839,5.66336080299) --(0.347993402996,5.64637365494) --(0.348575812014,5.66216466155) --(0.350176261507,5.7003918285) --(0.365904806452,5.70200878035) --(0.392039222893,5.68949602911) --(0.380752354633,5.62681268359) --(0.400595231733,5.62266035058) --(0.445936652131,5.60120262996) --(0.436450949276,5.63790818198) --(0.460598452659,5.66338130729) --(0.494043971244,5.67463863882) --(0.527040851898,5.66765976086) --(0.523193295112,5.6916068329) --(0.483314653637,5.67578952537) --(0.473062965342,5.72048117323) --(0.495581790222,5.73774335075) --(0.560228943107,5.76300594879) --(0.543394625882,5.73494485266) --(0.540427302668,5.74451065939) --(0.552255209824,5.70878968785) --(0.607746819646,5.70161344115) --(0.639950123106,5.68363600215) --(0.579695777256,5.65510204126) --(0.565944437538,5.67688256389) --(0.593279124683,5.65937113814) --(0.611872444348,5.67261597439) --(0.617182183842,5.64431883293) --(0.59478934008,5.67010935163) --(0.591014725846,5.68432615876) --(0.56490297537,5.65205878404) --(0.557298879461,5.69257800692) --(0.58599504505,5.65589328833) --(0.580229576456,5.64942888028) --(0.547535528192,5.61887324957) --(0.542730118795,5.62799160887) --(0.508262952173,5.63328967225) --(0.524318744116,5.64944416145) --(0.529957813967,5.66897161826) --(0.540503058454,5.6909766295) --(0.528947329905,5.7373459204) --(0.51697681751,5.70317222662) --(0.483564691641,5.70429042139) --(0.432016120235,5.73274020097) --(0.436199945786,5.76989646635) --(0.457015334089,5.73582124188) --(0.448565561573,5.7589443949) --(0.440406524181,5.74356527331) --(0.424343043911,5.76256018267) --(0.421877901234,5.7648043053) --(0.42323106931,5.77620546183) --(0.382910889523,5.79367966516) --(0.397905388745,5.80477675194) --(0.416917373891,5.84255693288) --(0.4445154389,5.84171706106) --(0.411779874737,5.8821019757) --(0.443366789584,5.92119452672) --(0.477310372121,5.98086073991) --(0.495762965365,5.983212639) --(0.504431165709,5.96328723314) --(0.501123220456,6.0088166527) --(0.506398359181,6.00854245777) --(0.496682636215,5.95428536815) --(0.50954265438,5.95936069857) --(0.495059928003,5.99135866341) --(0.471285707372,5.97132089607) --(0.524920912725,5.9664161041) --(0.508780363189,6.01701027875) --(0.508754972703,5.98820279129) --(0.495085398178,5.97022864992) --(0.53116059421,5.9798433308) --(0.570674305865,6.00544135652) --(0.561484820696,5.98706452773) --(0.602532758172,5.94713891483) --(0.585654228239,5.95525541755) --(0.590383328168,5.97304005438) --(0.626996041449,5.91496391173) --(0.608465207194,5.90885293169) --(0.630233304096,5.89328011567) --(0.670889141368,5.91503257326) --(0.707047305174,5.92808052508) --(0.693164622739,5.89367541097) --(0.71488356915,5.87619527614) --(0.699502142353,5.89090744752) --(0.719464355659,5.92771649544) --(0.679102145081,5.92513682043) --(0.657498014144,5.91961359244) --(0.700461717062,5.88485067803) --(0.673248755487,5.90697156638) --(0.693810788269,5.93107830438) --(0.65547361269,5.91003873989) --(0.669793659809,5.91113056485) --(0.653206561098,5.95344577452) --(0.666127707226,5.91700255893) --(0.661223929835,5.95949048361) --(0.677702457394,5.91716385002) --(0.682569604204,6.01308168823) --(0.741865148406,5.9747220866) --(0.751724251712,6.00952224015) --(0.762945776044,6.0197682954) --(0.783807976238,5.98627805872) --(0.77992894336,5.98518836645) --(0.756289424114,5.98002865077) --(0.783641814077,6.02151566531) --(0.810165707682,6.06808997106) --(0.774258640506,6.10917915575) --(0.774714659489,6.12715259206) --(0.813948939248,6.18309660259) --(0.81458132524,6.16180491508) --(0.822117573241,6.18537245853) --(0.833736097607,6.2004911165) --(0.869229117229,6.14323019316) --(0.864390096546,6.16251040623) --(0.853865675959,6.20618163013) --(0.848604792948,6.17190716148) --(0.818928068971,6.1732438636) --(0.839574712705,6.17438366746) --(0.803640655645,6.20877394026) --(0.78844359668,6.25173142398) --(0.796525494176,6.25826556607) --(0.822164051394,6.27240894621) --(0.841908027055,6.27154215445) --(0.804424115235,6.26268498144) --(0.792079087931,6.33927982628) --(0.788624072981,6.35004119563) --(0.75902944259,6.39371232533) --(0.768957908714,6.41443334527) --(0.778364880842,6.42617507392) --(0.8110343583,6.43890640464) --(0.779093499315,6.48035448158) --(0.813670393189,6.50432967772) --(0.866339429682,6.4952973904) --(0.904390821526,6.46578073972) --(0.919841278589,6.51180957048) --(0.96954653538,6.49387118399) --(1.00837777759,6.48953146265) --(1.00698698796,6.50249105639) --(1.02087773246,6.47086310431) --(1.01789539905,6.47003225584) --(0.994446468115,6.48820347108) --(0.972732588474,6.52652130886) --(0.940303729764,6.4978409932) --(0.928889505865,6.52084031476) --(0.901112003058,6.53552224961) --(0.888748493911,6.54059347536) --(0.885601573544,6.60321171389) --(0.94434910151,6.58586921613) --(0.981453032484,6.58759683906) --(1.00987686238,6.56298498402) --(1.03009660876,6.53939164294) --(1.06916875829,6.56412137808) --(1.02733958792,6.59455660733) --(1.08658401626,6.57674096841) --(1.11106821199,6.59680810885) --(1.13947465576,6.63757814158) --(1.17818206157,6.66795776307) --(1.15287445542,6.63292996816) --(1.20931483162,6.62322409782) --(1.25070642102,6.63291308928) --(1.22310901289,6.65247451792) --(1.21007548977,6.64729488446) --(1.20192471832,6.72935540206) --(1.23504365104,6.72899160912) --(1.24969211891,6.71285021934) --(1.20981804374,6.73078470167) --(1.22635420387,6.73530804335) --(1.20873702679,6.72649525184) --(1.21218984368,6.76468113563) --(1.24665915365,6.79964198629) --(1.24530488049,6.78802113357) --(1.2070781875,6.81581280018) --(1.19170262456,6.82454099454) --(1.23773798214,6.82096195787) --(1.24159012271,6.83106539456) --(1.21876012546,6.83600485878) --(1.23387630626,6.80288609049) --(1.22755427647,6.82341384863) --(1.25241541297,6.83390441753) --(1.25524337884,6.83463851165) --(1.27099553139,6.80113607854) --(1.20529225161,6.765319235) --(1.21153360157,6.74468066582) --(1.19136406512,6.69158353427) --(1.22558607481,6.70917248116) --(1.26471643189,6.74097344685) --(1.27402362388,6.74338641439) --(1.27001792918,6.74698457337) --(1.22474605485,6.78111189275) --(1.22229776845,6.8009065064) --(1.21733743938,6.79713934193) --(1.20869136766,6.82219129545) --(1.19184925141,6.87811329325) --(1.16585457241,6.88614966717) --(1.17832156566,6.8407260311) --(1.18636770368,6.84734381032) --(1.14423756848,6.85373162093) --(1.13267898019,6.86156899656) --(1.08682177227,6.82126662303) --(1.11793698826,6.84907867849) --(1.11137899998,6.87045452314) --(1.12168806578,6.89648054189) --(1.10982427281,6.96298558606) --(1.11072844317,6.97753904585) --(1.18449518876,6.98297010276) --(1.19428494043,7.01421183915) --(1.17436550395,7.07747366886) --(1.14249641231,7.13143589071) --(1.14356266485,7.10860570484) --(1.12237450161,7.10810873275) --(1.10353209397,7.1626398249) --(1.11713540992,7.16539024076) --(1.13126419554,7.17721825242) --(1.13317107273,7.18726852902) --(1.11515754447,7.14173847127) --(1.10962018206,7.16092548611) --(1.11234235935,7.1662881708) --(1.10177086702,7.1363293616) --(1.12925823621,7.12525015672) --(1.1087815371,7.13981074977) --(1.08004709022,7.15789187284) --(1.05072831344,7.17588084164) --(1.09286458511,7.22790696806) --(1.08884755635,7.24053533009) --(1.06326934055,7.25844865326) --(1.03592782532,7.26005481245) --(1.03102678157,7.27066566455) --(1.03696003035,7.30110235973) --(1.02685097598,7.3048245056) --(1.03723966688,7.35188730612) --(1.05199772669,7.35301434615) --(1.07685270615,7.39471027918) --(1.11644250134,7.39634174179) --(1.09575623844,7.39679949727) --(1.09378784749,7.39739110995) --(1.10402371199,7.42774189198) --(1.10707789197,7.45290510955) --(1.12633334405,7.49958243072) --(1.12092061159,7.510085263) --(1.12945873934,7.50706598316) --(1.10311179725,7.49219920116) --(1.12264614255,7.50893001714) --(1.17832021358,7.4532177736) --(1.18819329674,7.41955621257) --(1.23794649807,7.44811808274) --(1.26538526616,7.43878039792) --(1.26090601585,7.49627441729) --(1.2854410248,7.45947029298) --(1.32593878609,7.43041925059) --(1.31685965484,7.44165026741) --(1.32462936545,7.49164185617) --(1.36752854874,7.47699689139) --(1.36372080593,7.49734825527) --(1.35574011422,7.48112913229) --(1.3490667009,7.51672417852) --(1.38023450154,7.49753396957) --(1.37363871504,7.52245276097) --(1.36928145174,7.49388624562) --(1.34480163187,7.48932553138) --(1.3504519511,7.49721665513) --(1.35393551051,7.51998949128) --(1.36195996329,7.55482543218) --(1.37858650119,7.54103975252) --(1.31700045307,7.53741910356) --(1.35177275463,7.58404363085) --(1.38692141002,7.55688496175) --(1.33120960746,7.60731029229) --(1.34152122037,7.58137662418) --(1.34135727273,7.5668548574) --(1.38023815611,7.53956413282) --(1.39596032799,7.54891656135) --(1.37257436705,7.54892524902) --(1.3703890203,7.56234792368) --(1.34804840936,7.54783933839) --(1.37183375473,7.57214960582) --(1.42661524763,7.57598965321) --(1.46506572099,7.58803836738) --(1.48167063718,7.59131713444) --(1.52156249307,7.58851325862) --(1.5266048991,7.58133295511) --(1.51860794846,7.59735552306) --(1.49662967488,7.59562130906) --(1.42954380856,7.62162801899) --(1.48368093883,7.65572814071) --(1.49407019601,7.68983630077) --(1.51765137019,7.75771407917) --(1.50817770833,7.77438951211) --(1.5241057973,7.82334169376) --(1.54103895833,7.8064186774) --(1.5152235217,7.81142455001) --(1.53999565255,7.74974081908) --(1.52219519794,7.75547656598) --(1.4563423663,7.74865984317) --(1.47354569159,7.72003028605) --(1.4657207809,7.74905950281) --(1.55314167542,7.78175325149) --(1.51534214364,7.88424967541) --(1.54102737057,7.85587580541) --(1.54373509788,7.80585994468) --(1.57012264226,7.79974691677) --(1.5398159111,7.7780183365) --(1.51356967955,7.77384843097) --(1.55980288883,7.82228359837) --(1.54485467646,7.78789626247) --(1.53078047833,7.81873872134) --(1.54018411429,7.80102137123) --(1.57958231755,7.80458965705) --(1.60960826177,7.80445491963) --(1.65438944504,7.81122111058) --(1.68377062906,7.84119726996) --(1.68875339836,7.85166321696) --(1.69239611937,7.82411661059) --(1.65589080495,7.86928544124) --(1.67544959576,7.87107259354) --(1.67611800973,7.8543601626) --(1.71596352576,7.8517994697) --(1.7706730925,7.86172480788) --(1.75789290881,7.90450034276) --(1.80693326777,7.91529996514) --(1.76687117317,7.93649682039) --(1.80560204317,7.99711646048) --(1.8005910386,8.00958062014) --(1.83816476061,8.04386273109) --(1.8707338624,8.03238310813) --(1.8779567685,8.08394172511) --(1.86467775629,8.09999616689) --(1.83828123448,8.12937957333) --(1.82758675474,8.09658204004) --(1.79550600143,8.10591998049) --(1.85066065037,8.15884417083) --(1.85681592729,8.18320072741) --(1.89619380967,8.20079988314) --(1.85355822688,8.24122765383) --(1.91468959796,8.28876330914) --(1.96653934287,8.27821625387) --(2.03136247739,8.26138992441) --(1.98577805335,8.24565776722) --(2.00090793494,8.28783316324) --(2.03787442467,8.22704334316) --(2.07879714606,8.25179117999) --(2.08817658002,8.26955375842) --(2.0726731197,8.306503638) --(2.02451958385,8.28547894418) --(1.96940856225,8.26840845847) --(2.03338777984,8.28496547895) --(2.01025896286,8.32895486396) --(1.98094522426,8.3297390665) --(1.99980441483,8.34232410386) --(2.01293854032,8.31860038051) --(2.05638345227,8.32575721457) --(2.07115106542,8.33235610275) --(2.06334222638,8.28159448706) --(2.07163999627,8.28689925434) --(2.08099699687,8.30065163895) --(2.05873908438,8.32496005261) --(2.0256151127,8.32328574978) --(2.029314238,8.35128099132) --(2.06199736343,8.37788478465) --(2.04570877667,8.38764636051) --(2.08589888486,8.39706305782) --(2.1310410224,8.43470905866) --(2.16011900324,8.44011864778) --(2.15143016692,8.42512165696) --(2.18465691658,8.44057982092) --(2.18023541451,8.43801170402) --(2.22251286696,8.43640574291) --(2.20939873567,8.43603198515) --(2.25431127376,8.47251970746) --(2.29602111517,8.42812189925) --(2.34391072616,8.42266493207) --(2.34322086309,8.42926939095) --(2.35326528228,8.44871852265) --(2.37594445152,8.45943773604) --(2.38956437309,8.4629946046) --(2.40899998458,8.40514464742) --(2.48221077261,8.38540564684) --(2.46810698815,8.40324244516) --(2.53501341235,8.3929900407) --(2.55112781655,8.39780987486) --(2.5648990495,8.42303602712) --(2.51283996427,8.46905267126) --(2.568134095,8.50168293478) --(2.53287848391,8.45030918595) --(2.52745433722,8.50727622935) --(2.50869338838,8.49110139212) --(2.53443517355,8.52654986988) --(2.52919585758,8.53943430798) --(2.50870915562,8.54009238366) --(2.50035946806,8.55322116525) --(2.47735613142,8.52304522657) --(2.47568906604,8.5678380904) ;
\draw (-0.129688265081,5.46025055157) --(-0.121093959506,5.4610665913) --(-0.0875305333138,5.48816249465) --(-0.0791776573285,5.53110406774) --(-0.120435369889,5.54568338949) --(-0.125457829546,5.55469397528) --(-0.170203336641,5.56521004475) --(-0.121852333646,5.57267984134) --(-0.149688736761,5.53631038352) --(-0.183487927874,5.5326738338) --(-0.134806753925,5.48605165105) --(-0.107730544753,5.4508774433) --(-0.109410161337,5.45734746831) --(-0.125655839616,5.48005691628) --(-0.137708610492,5.50979140244) --(-0.15842282125,5.50892594681) --(-0.185410318757,5.52839303481) --(-0.21663347143,5.48264204221) --(-0.281633086317,5.48063638905) --(-0.317614873893,5.43383734257) --(-0.292215141373,5.38985361333) --(-0.264181573923,5.34037748436) --(-0.290364152339,5.33754823974) --(-0.314750864319,5.30612061186) --(-0.356047989062,5.33945234653) --(-0.373687068182,5.3099335655) --(-0.414185653039,5.26873997246) --(-0.40941110652,5.27104891508) --(-0.457491216243,5.2391064273) --(-0.492925184221,5.2488391162) --(-0.505667001079,5.24278955695) --(-0.506199177289,5.237003131) --(-0.546160277545,5.29053488141) --(-0.535414490119,5.30719607084) --(-0.586499663452,5.28350521621) --(-0.577316526839,5.30616216815) --(-0.53927626732,5.32312723851) --(-0.562296072516,5.33633488248) --(-0.60588304538,5.41819812459) --(-0.641845980434,5.45234478335) --(-0.674498845541,5.46062085238) --(-0.712751896309,5.43995539096) --(-0.758459565492,5.44968094912) --(-0.770743592291,5.42830844812) --(-0.794374798838,5.45144704902) --(-0.762486876518,5.41160325269) --(-0.790568019678,5.43140283398) --(-0.766184443674,5.47758515687) --(-0.757055026768,5.51346029981) --(-0.806834633344,5.53159246653) --(-0.844065608967,5.56500039453) --(-0.886345691261,5.62674827461) --(-0.895755515712,5.61416460026) --(-0.929664830077,5.64477582546) --(-0.887426887312,5.6784359576) --(-0.859475102732,5.69457541046) --(-0.897856297716,5.70371256237) --(-0.920621714291,5.75202489937) --(-0.939573049315,5.72352576622) --(-0.935890373405,5.75961837806) --(-0.98306389014,5.77599699546) --(-0.986064134928,5.80991074335) --(-1.03585148227,5.83195210496) --(-1.05801407729,5.82651384017) --(-1.07346895597,5.79710793134) --(-1.10336892619,5.79623002708) --(-1.12594530185,5.79067293953) --(-1.12329935127,5.78910112744) --(-1.10266665511,5.78387907862) --(-1.1151517033,5.83441876876) --(-1.1218155129,5.85112979126) --(-1.10675369529,5.88932510602) --(-1.13250880725,5.90241404982) --(-1.15186037261,5.89538843796) --(-1.14307402181,5.87738470394) --(-1.19964421498,5.87992205259) --(-1.20703935891,5.90068886989) --(-1.17995425174,5.88812454079) --(-1.21503049997,5.85742822732) --(-1.29407534607,5.83621463425) --(-1.30707849159,5.85609681895) --(-1.30469945323,5.86146743102) --(-1.2919125195,5.81792778268) --(-1.26817785618,5.81592407725) --(-1.2612936816,5.79843199016) --(-1.27354208671,5.7827121938) --(-1.24883272369,5.77327695022) --(-1.24725014523,5.80458309763) --(-1.20460491426,5.7955915079) --(-1.2385783082,5.8359535871) --(-1.20001474545,5.91811481323) --(-1.26621902641,5.9218600118) --(-1.24492813808,5.89533244471) --(-1.25491397097,5.9246050732) --(-1.22797511306,5.92630841023) --(-1.24548354957,5.98207004055) --(-1.31582885684,6.0080488572) --(-1.3099748091,5.99855157435) --(-1.2898402447,6.05716491393) --(-1.30421266728,6.06839094564) --(-1.35769244113,6.03020095753) --(-1.35407148671,5.99750301377) --(-1.40160941443,5.97878565577) --(-1.40532968574,6.04375926675) --(-1.40776658881,6.02613704945) --(-1.40804371453,6.06244834036) --(-1.38392502151,6.07166391676) --(-1.42471245941,6.10225020277) --(-1.43248345977,6.07034183246) --(-1.41493164586,6.10289190737) --(-1.42088143431,6.11121587576) --(-1.41137775299,6.11050450793) --(-1.36418019907,6.11661007483) --(-1.33972959639,6.1202846116) --(-1.31519482415,6.16477053492) --(-1.34121322708,6.16782954687) --(-1.32749271331,6.1714748923) --(-1.29323665125,6.17685502514) --(-1.24213339697,6.19943260021) --(-1.2463358264,6.17392859171) --(-1.29603021169,6.18430492509) --(-1.2315721115,6.16330019907) --(-1.25973392814,6.14886350975) --(-1.2554599439,6.15832638547) --(-1.2601165812,6.19649714663) --(-1.23191734604,6.19006606374) --(-1.23887484776,6.19753058656) --(-1.21184663899,6.12371995301) --(-1.19120170333,6.09084612442) --(-1.22084381984,6.13172993181) --(-1.24907760157,6.14068180068) --(-1.23378479925,6.14849098016) --(-1.23506924115,6.1264919841) --(-1.22200374896,6.1280203437) --(-1.20289165754,6.17971423536) --(-1.18612251645,6.1803613348) --(-1.1253851822,6.12809479154) --(-1.14485192482,6.14480668879) --(-1.1463962764,6.17058837992) --(-1.1643162502,6.14946280594) --(-1.22180901146,6.23506872363) --(-1.26993340832,6.20854614692) --(-1.26716594464,6.23014761926) --(-1.22451377305,6.22280460776) --(-1.24079394436,6.2508879679) --(-1.21967973229,6.27330586496) --(-1.23369283113,6.23989367805) --(-1.25253665305,6.21463605751) --(-1.24892137657,6.20972723428) --(-1.23912100446,6.20388159309) --(-1.28421551934,6.21360744394) --(-1.23593645302,6.23389349625) --(-1.2502492576,6.24947583814) --(-1.28032900546,6.25522907099) --(-1.30703492174,6.24884198967) --(-1.33695985941,6.24609651805) --(-1.30867441693,6.23579029949) --(-1.30504980661,6.23523011764) --(-1.36040725416,6.25966704429) --(-1.38008797033,6.24387827557) --(-1.40121476705,6.28473912091) --(-1.42198370443,6.28286786242) --(-1.46163310961,6.34719268386) --(-1.49977463266,6.33323595135) --(-1.48504842145,6.33830680224) --(-1.53324858518,6.273966711) --(-1.54756734427,6.26865490387) --(-1.49306140219,6.26559093292) --(-1.58641050877,6.28714784028) --(-1.58161631538,6.32752762893) --(-1.57342888736,6.32737103126) --(-1.58555436037,6.34491546402) --(-1.56519971777,6.32844380806) --(-1.63143045316,6.29935662333) --(-1.64149958446,6.2339379772) --(-1.67486307461,6.21066651414) --(-1.71848685379,6.15401355598) --(-1.72057150764,6.15268552469) --(-1.77663812761,6.10176432368) --(-1.81768173868,6.09915556029) --(-1.83149979755,6.05125914765) --(-1.83995688991,6.12700343373) --(-1.83994699184,6.07548161624) --(-1.85546659478,6.10308951077) --(-1.83039945022,6.16535313568) --(-1.84598738547,6.11691166289) --(-1.8159444556,6.12459637607) --(-1.84505595577,6.14038215806) --(-1.78100835834,6.1754502208) --(-1.78934293841,6.1654399097) --(-1.77411887463,6.22634831861) --(-1.77758327402,6.22396099051) --(-1.78097553889,6.20932399041) --(-1.77739145135,6.23522151143) --(-1.85000173434,6.24883643929) --(-1.81657555553,6.26958563138) --(-1.82164489513,6.21523148462) --(-1.80457562509,6.20857984301) --(-1.7702548684,6.25639081951) --(-1.74147153749,6.24963380163) --(-1.75023083028,6.26059858073) --(-1.75391562685,6.23498020827) --(-1.76672004519,6.23527047715) --(-1.78875159744,6.2608749612) --(-1.82074911255,6.24322253489) --(-1.8068216592,6.21947699583) --(-1.78153012565,6.22243401792) --(-1.75025051243,6.2340697193) --(-1.75397178869,6.28876064363) --(-1.74909035684,6.32581100147) --(-1.73853049758,6.32018761358) --(-1.79371134524,6.27227457831) --(-1.79165410291,6.30716561776) --(-1.81495390786,6.33985925799) --(-1.85548496973,6.34760606298) --(-1.8566925845,6.35208308011) --(-1.8511121544,6.37913151089) --(-1.86777232426,6.37972062964) --(-1.86387417985,6.36129169508) --(-1.9081624407,6.37672494349) --(-1.98177278695,6.41708831092) --(-1.98465280663,6.43658570213) --(-2.04870995651,6.4958556103) --(-2.01894610566,6.53550482236) --(-2.03148536238,6.51521882894) --(-2.02527865165,6.47219322433) --(-2.00715511533,6.53264149733) --(-2.01581834007,6.59236210205) --(-1.99306093219,6.56886298602) --(-2.00828183749,6.54831582691) --(-2.0606928008,6.5494758634) --(-2.02974029001,6.53764996615) --(-2.00069227961,6.55896975392) --(-1.98723229219,6.60569606301) --(-2.01197140468,6.60297974664) --(-1.91650479866,6.70021952976) --(-1.87664966382,6.71018716788) --(-1.88850132944,6.68693324704) --(-1.8758653133,6.66737116602) --(-1.9066786129,6.69127658113) --(-1.89026035033,6.70873796511) --(-1.83589204603,6.70010885509) --(-1.83928917662,6.66683229408) --(-1.86439338771,6.64906743208) --(-1.8764296843,6.66437895154) --(-1.87400934487,6.64257044074) --(-1.91419305485,6.64085009462) --(-1.923795256,6.64279986344) --(-1.94802081547,6.62266772897) --(-1.91338941023,6.61742606964) --(-1.92606295112,6.60802703243) --(-1.87791900376,6.61131018925) --(-1.86969582408,6.59034299413) --(-1.87327646097,6.55152225024) --(-1.86086507354,6.54383277491) --(-1.79437129482,6.52637736139) --(-1.81286695262,6.53358614722) --(-1.81481372552,6.56452006625) --(-1.84315401605,6.58504000807) --(-1.85493775852,6.51274956691) --(-1.82280429693,6.55038932762) --(-1.84633931623,6.54796668913) --(-1.81475608305,6.52965300751) --(-1.82577117221,6.48371673923) --(-1.84456177445,6.50824969388) --(-1.87960657953,6.51698261142) --(-1.90891862192,6.5049901642) --(-1.95337807309,6.54867263637) --(-1.93931600429,6.52840793865) --(-1.94363593102,6.56520810397) --(-1.92112922197,6.59305643667) --(-1.93024057631,6.54144887837) --(-1.8957192037,6.58212552721) --(-1.87577562518,6.58826051694) --(-1.83837400092,6.61346077046) --(-1.817256427,6.60651397728) --(-1.79538009349,6.60459939868) --(-1.74512090562,6.61781646893) --(-1.73317998838,6.6156863299) --(-1.75169917503,6.61796485919) --(-1.70302570175,6.59883805942) --(-1.69952248635,6.58049389658) --(-1.71501183862,6.55654490603) --(-1.71072098774,6.53755652768) --(-1.68823188667,6.49230575536) --(-1.70895262403,6.49459328902) --(-1.72542898016,6.4833389689) --(-1.73606702469,6.4975756975) --(-1.72548140462,6.52801984759) --(-1.67303383215,6.56865613913) --(-1.72781754061,6.58041676494) --(-1.71921046269,6.5705131958) --(-1.68783462161,6.5683559209) --(-1.69558188832,6.60863598469) --(-1.68307958501,6.62768271844) --(-1.71313737863,6.60258082756) --(-1.7632991204,6.55108830666) --(-1.7195310162,6.54232759695) --(-1.74360236719,6.5258467773) --(-1.76305698986,6.58058090334) --(-1.79132234363,6.57965149608) --(-1.8218373049,6.59229685847) --(-1.86289654286,6.59823383168) --(-1.87346193643,6.58854748979) --(-1.844904681,6.62781397774) --(-1.87348278551,6.62862454571) --(-1.87745947789,6.62058549796) --(-1.90025516231,6.64924699252) --(-1.92800767764,6.63943308201) --(-1.91040500389,6.68323662635) --(-1.92949200354,6.68789576829) --(-1.91733406887,6.66468557047) --(-1.93742820572,6.64801644162) --(-1.95216733557,6.5988493461) --(-1.92167234506,6.64978444246) --(-1.91174247591,6.63928692619) --(-1.95552571656,6.639755164) --(-1.9177012486,6.62148466874) --(-1.92059648357,6.60493386266) --(-1.92805513893,6.62323127375) --(-1.92864345776,6.61603570379) --(-1.97262676214,6.66883877005) --(-2.05195916762,6.63529316495) --(-2.05249670911,6.67332270951) --(-2.01860188166,6.6748787082) --(-1.95850739466,6.66405476274) --(-1.92729284054,6.70057753879) --(-1.93293605744,6.70175126101) --(-1.95468694555,6.74547477081) --(-2.03038094516,6.76637623999) --(-2.03767718222,6.78560353938) --(-2.04623509405,6.77626713481) --(-2.08975438986,6.80323456194) --(-2.10523131273,6.77716071943) --(-2.16705368342,6.7853393086) --(-2.16572915818,6.76948208316) --(-2.12056218413,6.78341489484) --(-2.15210988943,6.79013542762) --(-2.17408777068,6.82718868779) --(-2.18202746928,6.83570164161) --(-2.19628138481,6.77231402252) --(-2.2436660962,6.77137787003) --(-2.21799992876,6.82022522131) --(-2.19625030065,6.83799424026) --(-2.22378558724,6.84474035721) --(-2.21798678618,6.84666480402) --(-2.23590398736,6.84924134335) --(-2.25328190395,6.87671222611) --(-2.23449065715,6.92373491816) --(-2.26521807739,6.93323228105) --(-2.27238639192,6.92687016637) --(-2.2610046632,6.96118297391) --(-2.33806803127,6.94191854405) --(-2.31124315691,6.95778581861) --(-2.33688788488,6.98670716695) --(-2.36503114356,6.96421972487) --(-2.35540375062,6.97044496713) --(-2.34488415567,6.96192096882) --(-2.32248276092,6.91911207394) --(-2.35156241664,6.91839024579) --(-2.31273337314,6.93203435103) --(-2.31745067972,6.91862119751) --(-2.31404194021,6.94442653652) --(-2.30962615186,6.9195649184) --(-2.27589608501,6.93958311891) --(-2.28162535969,6.89585897389) --(-2.2501712344,6.89397365332) --(-2.27126059428,6.89136982209) --(-2.25037071581,6.87670320098) --(-2.25456773149,6.88007172088) --(-2.24443083102,6.86177248653) --(-2.21226899763,6.91311321881) --(-2.16441526149,6.93560199554) --(-2.18908335218,6.90134080476) --(-2.22521525024,6.95157183481) --(-2.25351232745,6.95599814567) --(-2.29448284443,6.98718100857) --(-2.29229236388,6.99312153547) --(-2.3367610502,6.95063703514) --(-2.34694795191,6.9381896455) --(-2.3209281965,6.97392786094) --(-2.36601261068,7.01229177126) --(-2.35271443424,7.11026638102) --(-2.39793920949,7.11355698004) --(-2.35049660202,7.08566894315) --(-2.3429284054,7.06772246525) --(-2.33750484187,7.08173948756) --(-2.34668439087,7.09583329078) --(-2.32550609234,7.11676234711) --(-2.32159136972,7.14729045954) --(-2.33321808493,7.15077601453) --(-2.33360235572,7.16426533913) --(-2.29929967892,7.12305988973) --(-2.33312440882,7.18334052348) --(-2.37332848283,7.21536559354) --(-2.34407766282,7.2671058877) --(-2.38157080036,7.30454539878) --(-2.39453808814,7.33350829192) --(-2.39790881515,7.36030009712) --(-2.40862884914,7.37768167803) --(-2.40070614522,7.41719677773) --(-2.3988268612,7.36667397046) --(-2.38824401269,7.38904704586) --(-2.38743060419,7.36160866504) --(-2.38810978354,7.36997518624) --(-2.41345887156,7.34115956002) --(-2.37543483182,7.32155309012) --(-2.37318139346,7.31856359113) --(-2.37353801726,7.31048074399) --(-2.38764236307,7.29899729506) --(-2.40495422324,7.29957662248) --(-2.39845462733,7.32060614459) --(-2.389473815,7.33870367076) --(-2.38414148116,7.36253387237) --(-2.41072353781,7.38170268527) --(-2.41410042686,7.40705210771) --(-2.42567160552,7.34917606229) --(-2.43216451362,7.34555876901) --(-2.42858738494,7.36953888518) --(-2.50755147761,7.36404313909) --(-2.5235975084,7.38499874215) --(-2.56995130315,7.37480275714) --(-2.57739977568,7.33385071541) --(-2.59332653889,7.31378429318) --(-2.61962320046,7.32036994221) --(-2.6639700223,7.34896163515) --(-2.66093582016,7.41932668177) --(-2.64198427181,7.44354534206) --(-2.63556079368,7.46416329889) --(-2.64560283662,7.42891178317) --(-2.65939353205,7.49123711348) --(-2.69414141571,7.50403994484) --(-2.70633950235,7.48976974603) --(-2.66787866118,7.56205300482) --(-2.68643044489,7.61631213827) --(-2.64223277071,7.58608986219) --(-2.63204497609,7.58253894491) --(-2.62360227639,7.59804968969) --(-2.61757618063,7.61767982513) --(-2.63226358107,7.56671976193) --(-2.60097054781,7.58682475792) --(-2.57769090432,7.59102072394) --(-2.58063741236,7.62279409819) --(-2.59908253204,7.66176423247) --(-2.62938503156,7.67751680115) --(-2.63822536028,7.68901503735) --(-2.65374022677,7.68992081013) --(-2.66104808424,7.7133631396) --(-2.67946623306,7.75918085451) --(-2.67743207235,7.71625976604) --(-2.72891283852,7.75689833423) --(-2.7191383121,7.77346324926) --(-2.71192325606,7.79204282323) --(-2.76723789079,7.72723835822) --(-2.75109048439,7.70863821562) --(-2.72782837578,7.74489180496) --(-2.70512232051,7.74240683462) --(-2.72799519834,7.77383466145) --(-2.74948029444,7.77932627568) --(-2.70098264835,7.79452605929) --(-2.72253494379,7.7981857508) --(-2.73040933225,7.81701700034) --(-2.74317929202,7.80983091793) --(-2.73282537753,7.8372440625) --(-2.66693302279,7.84150574261) --(-2.71100061427,7.8692325128) --(-2.70151291218,7.8637785375) --(-2.71704384021,7.89455645772) --(-2.72960432922,7.91735829698) --(-2.73746858309,7.86374353975) --(-2.756126047,7.9064609888) --(-2.76172518386,7.87034681999) --(-2.76551688995,7.8652130248) --(-2.78494420228,7.94729209989) --(-2.80470622331,7.94232471608) --(-2.83267640475,7.89273990349) --(-2.84059356989,7.89442189841) --(-2.855861061,7.85188679943) --(-2.82447171892,7.87464694428) --(-2.88423281816,7.89874285501) --(-2.88060875471,7.90821312347) --(-2.83123671724,7.91800976571) --(-2.78619590969,7.90138773061) --(-2.74034768744,7.88078232307) --(-2.70085648349,7.87530783492) --(-2.72113895275,7.87416815232) --(-2.73441337584,7.9248888375) --(-2.76674493766,7.94952666259) --(-2.76006535788,7.93195648679) --(-2.73392396862,7.938550081) --(-2.77101586907,7.94452203204) --(-2.77204238816,7.95544450852) --(-2.79625061529,7.93785410734) --(-2.77239891868,7.99043934246) --(-2.73163430658,8.04620658967) --(-2.71988148321,8.07598427222) --(-2.75791183118,8.06948724157) --(-2.77185681132,8.10105969162) --(-2.75014460457,8.14517506845) --(-2.78384770622,8.16713609952) --(-2.81715033269,8.18003600498) ;
\draw (-3.7,7.7) arc (125:55:7);
\draw (-3.7,7.7) arc (125:130:7);

\end{tikzpicture}
\caption{Path of a pair of individuals that reach the frontier of the branching Brownian motion.}
\end{figure}

\begin{proof}
We choose $C>0$ large enough such that for any $t \geq 1$, $y \in [1,t^{1/2}]$ and $v,v' \in \S^{d-1}$ verifying $\norm{v-v'}>Ct^{-1/2}$, we have $\left\{ x \in \R^d : \norm{x} \leq f^{t,y}_t, x.v \geq f^{t,y}_t - 1, x.v' \geq f^{t,y}_t - 1 \right\} = \emptyset$. Thus we assume in the rest of the proof that $\calA^{t,y}_v \cap \calA^{t,y}_{v'} = \emptyset$.

Let $v \neq v' \in \S^{d-1}$ be such that $v.v' \geq 0$ and $\norm{v-v'} > C t^{-1/2}$. We set $w = \frac{v+v'}{\norm{v+v'}}$, $w'=\frac{v-v'}{\norm{v-v'}}$ and $\theta \in [0,\pi/4]$ such that $v = w \cos(\theta) + w' \sin(\theta)$. Given $B$,$B'$ two independent Brownian motions and $s \leq t$ we set $W^s : r \mapsto B_{r \wedge s} + B'_{(r-s)_+}$.

We apply the many-to-two lemma, we obtain
\begin{equation}
  \label{eqn:manytotwoApplied}
  \E\left[ \# \calA^{t,y}_v \# \calA^{t,y}_{v'} \right] \leq \int_0^t e^{2t-s}\P\left[ \begin{array}{l} B_t.v \geq f^{t,y}_t-1, \norm{B_r} \leq f^{t,y}_r, r \leq t \\ W^s_t.v' \geq f^{t,y}_t - 1, \norm{W^s_r} \leq f^{t,y}_r, r \leq t\end{array} \right].
\end{equation}
Using the Markov property at time $s$, we have
\[
  e^{2t-s}\P\left[ \begin{array}{l} B_t.v \geq f^{t,y}_t-1, \norm{B_r} \leq f^{t,y}_r, r \leq t \\ W^s_t.v' \geq f^{t,y}_t - 1, \norm{W^s_r} \leq f^{t,y}_r, r \leq t\end{array} \right]\leq e^s \E\left[ \phi_{s,v}(B_s) \phi_{s,v'}(B_s) \ind{\norm{B_r} \leq f^{t,y}_r, r \leq s}\right],
\]
where, for $v \in \S^{d-1}$, $\phi_{s,v} : x \in \R^d \mapsto e^{t-s}\P\left[ (B_{t-s}+x).v \geq f^{t,y}_t-1, \norm{B_r+x} \leq f^{t,y}_{s+r}, r \leq t-s \right]$. For any $x \in \R^d$ such that $\norm{x} \leq f^{t,y}_s$, setting $\beta = B.v$, we have
\begin{align*}
  \phi_{s,v}(x) &\leq e^{t-s}\P\left[ \beta_{t-s} \geq f^{t,y}_t - 1 - x.v, \beta_r \leq f^{t,y}_{s+r}-x.v, r \leq t-s \right]\\
  &\leq \E\left[ e^{-\sqrt{2}\beta_{t-s}} \ind{\beta_{t-s} \geq \tilde{f}^{t,y}_t - 1 + \sqrt{2} s - x.v, \beta_r \leq \tilde{f}^{t,y}_{s+r} +\sqrt{2}s- x.v, r \leq t-s} \right]\\
  &\leq \frac{C}{(t+1)^{(d-4)/2}} e^{-\sqrt{2}y} e^{\sqrt{2}(x.v-\sqrt{2}s)} \frac{(f^{t,y}_s - x.v)_+}{(t-s+1)^{3/2}}, 
\end{align*}
by the Girsanov transform and \eqref{eqn:excursionBended}. Consequently, we have
\begin{multline*}
e^{2t-s}\P\left[  B_t.v \geq f^{t,y}_t-1, \norm{B_r} \leq f^{t,y}_r, r \leq t, W^s_t.v' \geq f^{t,y}_t - 1, \norm{W^s_r} \leq f^{t,y}_r, r \leq t \right]\\
\leq C \tfrac{e^{-2\sqrt{2}y}}{t^{d-4}(t-s+1)^3} e^s \E\left[ e^{\sqrt{2}(B_s.(v+v') - 2 \sqrt{2}s)} (f^{t,y}_s - B_s.v)_+(f^{t,y}_s - B_s.v')_+ \ind{\norm{B_r} \leq f^{t,y}_r, r \leq s} \right].
\end{multline*}

We observe that for any $x \in \R^d$
\[
  f^{t,y}_s - x.v =  f^{t,y}_s - x.w \cos(\theta) - x.w' \sin (\theta) = f^{t,y}_s(1 - \cos(\theta)) + \cos (\theta)(f^{t,y}_s-x.w) -  x.w' \sin (\theta).
\]
Moreover, for any $a,b \in \R$ we have $(a+b)_+(a-b)_+ \leq 2a^2$ (both when $|b|< a$ and $|b|\geq a$) therefore
\begin{align*}
  \left(f^{t,y}_s - x.v\right)_+\left(f^{t,y}_s-x.v'\right)_+ &\leq 2 \left(f^{t,y}_s(1 - \cos(\theta)) + \cos(\theta) \left( f^{t,y}_s-x.w \right) \right)^2\\
  &\leq 4 \left( \left(f^{t,y}_s(1 - \cos(\theta))\right)^2 + \cos(\theta)^2 \left( f^{t,y}_s-x.w \right)^2 \right).
\end{align*}
As $\theta \in [0,\pi/4]$, there exists $C>0$ such that $1 - \cos(\theta) \leq C \theta^2$, thus
\[
  \left(f^{t,y}_s - x.v\right)_+\left(f^{t,y}_s-x.v'\right)_+ \leq C \left( (s+y)^2 \theta^4 + \left( f^{t,y}_s-x.w \right)^2 \right),
\]
yielding
\begin{multline*}
  e^{2t-s}\P\left[ B_t.v \geq f^{t,y}_t-1, \norm{B_r} \leq f^{t,y}_r, r \leq t, W^s_t.v' \geq f^{t,y}_t - 1, \norm{W^s_r} \leq f^{t,y}_r, r \leq t \right]\\
  \leq \frac{C e^{-2\sqrt{2}y} e^s}{(t+1)^{d-4}(t-s+1)^3} \E\left[ e^{2 \sqrt{2}(\beta_s \cos(\theta) -  \sqrt{2}s)} \left[ (s+y)^2 \theta^4 + (f^{t,y}_s - \beta_s)^2 \right] \ind{\beta_r \leq f^{t,y}_r, r \leq s} \right].
\end{multline*}
We use once again the Girsanov transform, we have
\begin{multline*}
  e^s \E\left[ e^{2 \sqrt{2}(\beta_s \cos(\theta) -  \sqrt{2}s)} \left[ (s+y)^2 \theta^4 + (f^{t,y}_s - \beta_s)^2  \right] \ind{\beta_r \leq f^{t,y}_r, r \leq s} \right]\\
  = e^{4 s\left( \cos(\theta)-1\right)} \E\left[ e^{\sqrt{2} (2 \cos(\theta)-1)  \beta_s} \left[ (s+y)^2 \theta^4 + (\tilde{f}^{t,y}_s - \beta_s)^2 \right] \ind{\beta_r \leq \tilde{f}^{t,y}_r, r \leq s} \right].
\end{multline*}
For any $\theta < \frac{\pi}{4}$, we have $2\cos(\theta)-1 > 0.4$. Decomposing with respect to the value of $\beta_s$,
\begin{align*}
  &\E\left[ e^{\sqrt{2} (2 \cos(\theta)-1)  \beta_s} \left[ (s+y)^2 \theta^4 + (f^{t,y}_s - \beta_s)^2 \right] \ind{\beta_r \leq f^{t,y}_r, r \leq s} \right]\\
  \leq & C \sum_{k=0}^{+\infty} e^{\sqrt{2} (2 \cos(\theta)-1) (\tilde{f}^{t,y}_s- k)} \left(  (s+y)^2 \theta^4 + (k+1)^2 \right) \P\left[ \beta_s - \tilde{f}^{t,y}_s \in [-k-1,-k], \beta_r \leq \tilde{f}^{t,y}_r, r \leq s \right]\\
  \leq & \frac{Cy(s+y)^2\theta^4}{(s+1)^{3/2}}\frac{(s+y)^{(d-1)/2}(t-s+1)^{3/2}}{(t+1)^{3/2}} e^{\sqrt{2}y+ 2 \sqrt{2}(\cos(\theta)-1) y},
\end{align*}
using \eqref{eqn:excursionBended}. We conclude that for any $s \leq t$,
\begin{multline}
  \label{eqn:domi}
  e^{2t-s}\P\left[ B_t.v \geq f^{t,y}_t-1, \norm{B_r} \leq f^{t,y}_r, r \leq t, W^s_t.v' \geq f^{t,y}_t - 1, \norm{W^s_r} \leq f^{t,y}_r, r \leq t \right]\\
  \leq C\frac{ye^{-\sqrt{2}y}}{(t+1)^{(d-1)/2}} \frac{(t+1)^{3/2}}{(s+1)^{3/2}(t-s+1)^{3/2}} \frac{\theta^4(s+y)^{(d+3)/2}}{(t+1)^{(d-1)/2}} e^{-1.1 \theta^2 (s+y)}.
\end{multline}
Note that for any $\lambda > 0$,
\begin{multline*}
  \qquad \int_0^{t/2} \frac{(t+1)^{3/2}}{(s+1)^{3/2}(t-s+1)^{3/2}} \frac{(s+y)^{(d+3)/2}}{(t+1)^{(d-1)/2}} e^{-\lambda (s+y)} ds\\
  \leq \frac{\int_0^{+\infty} s^{d/2} e^{-\lambda s}ds}{(t+1)^{(d-1)/2}}
  \leq \frac{\Gamma(d/2+1)}{\lambda^{d/2+1}(t+1)^{(d-1)/2}}. \qquad
\end{multline*}
Moreover, for any $s > t/2$,
\[
  e^{2t-s}\P\left[ B_t.v \geq f^{t,y}_t-1, \norm{B_r} \leq f^{t,y}_r, r \leq t, W^s_t.v' \geq f^{t,y}_t - 1, \norm{W^s_r} \leq f^{t,y}_r, r \leq t \right] \leq C e^{-1.1 \theta^2 t}
\]

Therefore, \eqref{eqn:manytotwoApplied} and \eqref{eqn:domi} yield
\[
  \E\left[ \# \calA^{t,y}_v \# \calA^{t,y}_{v'} \right] \leq \frac{Cye^{-\sqrt{2}y}}{(t+1)^{(d-1)/2}} \left[ \frac{1}{\theta^{d-2}(t+1)^{(d-1)/2}} + e^{-\theta^2 t} \right],
\]
which ends the proof.
\end{proof}

The proof of Lemma \ref{lem:crossedmoment} hints that with high probability, two individuals $u,u'$ alive at time $t$ close to the frontier of the process such that $\norm{X_t(u)-X_t(u')} \geq Ct^{1/2}$ verify $\text{MRCA}(u,u') = o_\P(t)$. Mixing Lemmas \ref{lem:firstmoment}, \ref{lem:secondmoment} and \ref{lem:crossedmoment}, we bound from below $\P(R_t \geq f^{t,y}_t)$.

\begin{lemma}
\label{lem:lowerbound}
For any $\epsilon>0$, there exists $C>0$ such that for any $t \geq 1$, $v \in \S^{d-1}$ and $y \in [1,t^{1/2}]$ we have
\[
  \P\left( \exists u \in \calN_t : \norm{X_t(u)} \geq \sqrt{2} t + \frac{d-4}{2\sqrt{2}} \log t + y, \frac{X_t(u)}{\norm{X_t(u)}} . v > 1-\epsilon \right) \geq \frac{ye^{-\sqrt{2} y}}{C}.
\]
\end{lemma}

\begin{proof}
Let $v \in \S^{d-1}$, we set $v_2,\ldots v_d$ such that $(v,v_2,\ldots v_d)$ is an orthonormal basis of $\R^d$. Let $\epsilon>0$ and $t \geq 1$, we set
\[
  \calL_{t,\epsilon} = \left\{ w \in \S^{d-1} : w.v > 1 - \tfrac{\epsilon}{2},  \tfrac{t^{1/2}}{C} (w.v_j) \in \Z, j \in \{ 2,\ldots d\} \right\},
\]
where $C$ is a constant that we choose large enough such that Lemma \ref{lem:crossedmoment} holds. Note there exists $K>0$ such that $\frac{t^{(d-1)/2}}{K} \leq \# \calL_{t,\epsilon} \leq K t^{(d-1)/2}$. We observe that for any $t \geq 1$ large enough and $y \in [1,t^{1/2}]$, we have
\[
  \P\left( \exists u \in \calN_t : \norm{X_t(u)} \geq f^{t,y}_t-1, \frac{X_t(u)}{\norm{X_t(u)}} . v > 1-\epsilon \right) \geq \P\left[ \bigcup_{w \in \calL_{t,\epsilon}} \calA^{t,y}_w \neq \emptyset \right],
\]
and we bound this probability using the Cauchy-Scharz inequality. We have
\begin{equation}
  \P\left[ \bigcup_{w \in \calL_{t,\epsilon}} \calA^{t,y}_w \neq \emptyset \right]
  \geq \P\left[ \sum_{w \in \calL_{t,\epsilon}} \# \calA^{t,y}_w \geq 1\right]
  \geq \frac{\E\left[ \sum_{w \in \calL_{t,\epsilon}} \# \calA^{t,y}_w \right]^2}{\E\left[ \left( \sum_{w \in \calL_{t,\epsilon}} \# \calA^{t,y}_w \right)^2 \right]} \label{eqn:probab}.
\end{equation}

By Lemma \ref{lem:firstmoment}, we have
\[
  \E\left[ \sum_{w \in \calL_{t,\epsilon}} \# \calA^{t,y}_w \right] \geq \#\calL_{t,\epsilon} \frac{ y e^{-\sqrt{2} y}}{C t^{(d-1)/2}} \geq \frac{ye^{-\sqrt{2}y}}{C}.
\]
Similarly, using Lemma \ref{lem:secondmoment} we have $\E\left[ \sum_{w \in \calL_{t,\epsilon}} \left(\# \calA^{t,y}_w\right)^2 \right] \leq C y e^{-\sqrt{2} y}$. As $w.w' \geq Ct^{-1/2}$ for any $w \neq w' \in \calL_{t,\epsilon}$, we apply Lemma \ref{lem:crossedmoment} to compute
\begin{multline}
  \label{eqn:bigcrossmoment}
  \E\left[ \sum_{w \neq w' \in \calL_{t,\epsilon}}\# \calA^{t,y}_w \# \calA^{t,y}_{w'} \right]\\
  \leq \frac{C y e^{-\sqrt{2} y}}{t^{(d-1)/2}} \sum_{w \neq w' \in \calL_{t,\epsilon}} \left[ \frac{1}{\arccos(w.w')^{d-2}(t+1)^{(d-1)/2}} + e^{-\arccos(w.w')^2 t} \right].
\end{multline}
We observe there exists $C>0$ such that $\frac{\norm{w-w'}}{C} \leq \arccos(w.w') \leq C \norm{w-w'}$ for all $w, w' \in \calL_{\epsilon,t}$. Consequently, setting $\Z^{d-1}_t = \Z^{d-1} \cap \left[ -Ct^{1/2},Ct^{1/2} \right]^{d-1}$, \eqref{eqn:bigcrossmoment} becomes
\begin{multline*}
  \E\left[ \sum_{w \neq w' \in \calL_{t,\epsilon}}\# \calA^{t,y}_w \# \calA^{t,y}_{w'} \right]
  \leq \frac{C y e^{-\sqrt{2} y}}{t^{(d-1)/2}} \sum_{w \neq w' \in \calL_{t,\epsilon}} \left(\tfrac{1}{\norm{w-w'}^{d-2}(t+1)^{(d-1)/2}} + e^{-\norm{w-w'}^2 t} \right)\\
  \leq C y e^{-\sqrt{2} y} \sum_{(k_2,\ldots, k_d) \in \Z^{d-1}_t} \left(\tfrac{t^{(d-2)/2}}{ \left( \sum_{j=2}^d k_j^2 \right)^{(d-2)/2}(t+1)^{(d-1)/2}} + e^{-\left( \sum_{j=2}^d k_j^2 \right)} \right),
\end{multline*}
where $(k_2,\ldots k_d)$ are integers such that $(w-w').v_j = C t^{-1/2} k_j$. Note that
\[
  \sum_{(k_2,\ldots k_d) \in \Z^{(d-1)}_t} \frac{1}{\left(\sum_{j=2}^d k_j^2 \right)^{(d-2)/2}} \leq C t^{1/2} \quad \text{and} \quad \sum_{(k_2,\ldots k_d) \in \Z^{(d-1)}_t} e^{-\left( \sum_{j=2}^d k_j^2 \right)} \leq C,
\]
which yields $\E\left[ \sum_{w \neq w' \in \calL_{t,\epsilon}}\# \calA^{t,y}_w \# \calA^{t,y}_{w'} \right] \leq C ye^{-\sqrt{2}y}$.

By \eqref{eqn:probab}, we have
\[
  \P\left( \exists u \in \calN_t : \norm{X_t(u)} \geq f^{t,y}_t-1, \frac{X_t(u)}{\norm{X_t(u)}} . v > 1-\epsilon \right) \geq \frac{ye^{-\sqrt{2}y}}{C}.
\]
To conclude the proof, we observe that for any $t \geq 1$ large enough and $y \in [1,t^{1/2}]$, we have
\[
  f^{t,y-5}_t \leq \sqrt{2} t + \frac{d-4}{2\sqrt{2}} \log t + y \leq f^{t,y+5}_t.
\]
\end{proof}

\begin{proof}[Proof of Theorem \ref{thm:main}]
We set $r_t = \sqrt{2} t + \frac{d-4}{2\sqrt{2}} \log t$. The upper bound of Theorem \ref{thm:main} is a straightforward consequence of Lemma \ref{lem:tailUpperbound}. In effect
\[
  \lim_{y \to +\infty} \sup_{t \geq 0} \P\left( R_t \geq r_t + y \right) =0.
\]

The lower bound is obtained using a standard cutting argument. We observe the process $(\# \calN_t, t \geq 0)$ is a standard Yule process. In particular, for any $h > 0$, $\# \calN_h$ is a Geometric random variable with parameter $e^{-h}$. By Lemma \ref{lem:tailUpperbound}, we have $\P(R_h \geq \sqrt{2} h + h^{1/2}) \leq C h^{1/2} e^{-\sqrt{2} h^{1/2}}$. Applying the Markov property at time $h$, on the event $R_h \leq \sqrt{2} h + h^{1/2}$, the probability that $R_{t+h} \leq r_t - 2h$ is bounded from above by the probability that none of the $\# \calN_h$ individuals alive at time $h$ have a descendent that made a displacement greater that $r_t$. Thus
\begin{align*}
  \P(R_{t + h} \leq r_t - 2h) & \leq \P(R_h \geq \sqrt{2} h + h^{1/2}) + e^{-h}\sum_{j=0}^{+\infty} (1 - e^{-h})^j \P(R_t \leq r_t+1)^j\\
  &\leq C h^{1/2} e^{-\sqrt{2}h^{1/2}} + \frac{e^{-h}}{1 - (1 - e^{-h})\P(R_t \leq r_t+1)}.
\end{align*}
By Lemma \ref{lem:lowerbound}, $\sup_{t \geq 0} \P(R_t \leq r_t+1) < 1$, yielding $\lim_{h \to +\infty} \sup_{t \geq 0} \P(R_{t + h} \leq r_t - 2h) = 0$, which concludes the proof.
\end{proof}

\paragraph*{Acknowledgements.} I would like to thank Zhan Shi for its help at all stages of the research, Julien Berestycki for stimulating conversation and the referee for useful comments on the earlier version of the article.

\providecommand{\bysame}{\leavevmode\hbox to3em{\hrulefill}\thinspace}
\providecommand{\MR}{\relax\ifhmode\unskip\space\fi MR }
\providecommand{\MRhref}[2]{%
  \href{http://www.ams.org/mathscinet-getitem?mr=#1}{#2}
}
\providecommand{\href}[2]{#2}

\end{document}